\documentclass{tac}

\usepackage{enumerate,xspace}
\usepackage{amsmath,xspace,amssymb}
\usepackage[dvips]{graphics}
\usepackage{proof}
\usepackage{latexsym}  %For plain TeX symbols, such as \Box used in \qed
\usepackage{euscript}  %%Euler Script font

\input xy
\xyoption{all}
\xyoption{2cell}
\UseAllTwocells
\title{Connections \\ in Tangent Categories}
\author{J.R.B. Cockett and G.S.H. Cruttwell}
\thanks{Partially supported by NSERC Discovery Grants.}
\copyrightyear{2017}

\newtheorem{observation}{Remark}[section]
\newtheorem{lemma}[observation]{Lemma}  %%share counter with remark
\newtheorem{theorem}[observation]{Theorem}
\newtheorem{definition}[observation]{Definition}
\newtheorem{example}[observation]{Example}

\newtheorem{proposition}[observation]{Proposition} 
\newtheorem{corollary}[observation]{Corollary} 
 
\newcommand{\x}{\times}
\newcommand{\<}{\langle}
\renewcommand{\>}{\rangle}

\newcommand{\R}{\ensuremath{\mathbb R}\xspace}
\newcommand{\T}{\ensuremath{\mathbb T}\xspace}
\newcommand{\X}{\ensuremath{\mathbb X}\xspace}

\newcommand{\p}{\pi}
\newcommand{\zq}{0_{\sf q}}
\newcommand{\pq}{+_{\sf q}}
\newcommand{\pqone}{+_{\sf q1}}

\newcommand\nats{\hbox{$I \kern - .38em N$}} %Natural numbers
\newcommand\ints{\hbox{$Z \kern - .65em Z$}} %Integers

\makeatletter

\newdimen\w@dth

\def\setw@dth#1#2{\setbox\z@\hbox{\scriptsize $#1$}\w@dth=\wd\z@
\setbox\@ne\hbox{\scriptsize $#2$}\ifnum\w@dth<\wd\@ne \w@dth=\wd\@ne \fi
\advance\w@dth by 1.2em}

\def\t@^#1_#2{\allowbreak\def\n@one{#1}\def\n@two{#2}\mathrel
{\setw@dth{#1}{#2}
\mathop{\hbox to \w@dth{\rightarrowfill}}\limits
\ifx\n@one\empty\else ^{\box\z@}\fi
\ifx\n@two\empty\else _{\box\@ne}\fi}}
\def\t@@^#1{\@ifnextchar_ {\t@^{#1}}{\t@^{#1}_{}}}

\def\t@left^#1_#2{\def\n@one{#1}\def\n@two{#2}\mathrel{\setw@dth{#1}{#2}
\mathop{\hbox to \w@dth{\leftarrowfill}}\limits
\ifx\n@one\empty\else ^{\box\z@}\fi
\ifx\n@two\empty\else _{\box\@ne}\fi}}
\def\t@@left^#1{\@ifnextchar_ {\t@left^{#1}}{\t@left^{#1}_{}}}

\def\two@^#1_#2{\def\n@one{#1}\def\n@two{#2}\mathrel{\setw@dth{#1}{#2}
\mathop{\vcenter{\hbox to \w@dth{\rightarrowfill}\kern-1.7ex
                 \hbox to \w@dth{\rightarrowfill}}%
       }\limits
\ifx\n@one\empty\else ^{\box\z@}\fi
\ifx\n@two\empty\else _{\box\@ne}\fi}}
\def\tw@@^#1{\@ifnextchar_ {\two@^{#1}}{\two@^{#1}_{}}}

\def\tofr@^#1_#2{\def\n@one{#1}\def\n@two{#2}\mathrel{\setw@dth{#1}{#2}
\mathop{\vcenter{\hbox to \w@dth{\rightarrowfill}\kern-1.7ex
                 \hbox to \w@dth{\leftarrowfill}}%
       }\limits
\ifx\n@one\empty\else ^{\box\z@}\fi
\ifx\n@two\empty\else _{\box\@ne}\fi}}
\def\t@fr@^#1{\@ifnextchar_ {\tofr@^{#1}}{\tofr@^{#1}_{}}}

\newdimen\W@dth
\def\setW@dth#1#2{\setbox\z@\hbox{$#1$}\W@dth=\wd\z@
\setbox\@ne\hbox{$#2$}\ifnum\W@dth<\wd\@ne \W@dth=\wd\@ne \fi
\advance\W@dth by 1.2em}

\def\T@^#1_#2{\allowbreak\def\N@one{#1}\def\N@two{#2}\mathrel
{\setW@dth{#1}{#2}
\mathop{\hbox to \W@dth{\rightarrowfill}}\limits
\ifx\N@one\empty\else ^{\box\z@}\fi
\ifx\N@two\empty\else _{\box\@ne}\fi}}
\def\T@@^#1{\@ifnextchar_ {\T@^{#1}}{\T@^{#1}_{}}}

\def\T@left^#1_#2{\def\N@one{#1}\def\N@two{#2}\mathrel{\setW@dth{#1}{#2}
\mathop{\hbox to \W@dth{\leftarrowfill}}\limits
\ifx\N@one\empty\else ^{\box\z@}\fi
\ifx\N@two\empty\else _{\box\@ne}\fi}}
\def\T@@left^#1{\@ifnextchar_ {\T@left^{#1}}{\T@left^{#1}_{}}}

\def\Tofr@^#1_#2{\def\N@one{#1}\def\N@two{#2}\mathrel{\setW@dth{#1}{#2}
\mathop{\vcenter{\hbox to \W@dth{\rightarrowfill}\kern-1.7ex
                 \hbox to \W@dth{\leftarrowfill}}%
       }\limits
\ifx\N@one\empty\else ^{\box\z@}\fi
\ifx\N@two\empty\else _{\box\@ne}\fi}}
\def\T@fr@^#1{\@ifnextchar_ {\Tofr@^{#1}}{\Tofr@^{#1}_{}}}

\def\Two@^#1_#2{\def\N@one{#1}\def\N@two{#2}\mathrel{\setW@dth{#1}{#2}
\mathop{\vcenter{\hbox to \W@dth{\rightarrowfill}\kern-1.7ex
                 \hbox to \W@dth{\rightarrowfill}}%
       }\limits
\ifx\N@one\empty\else ^{\box\z@}\fi
\ifx\N@two\empty\else _{\box\@ne}\fi}}
\def\Tw@@^#1{\@ifnextchar_ {\Two@^{#1}}{\Two@^{#1}_{}}}

\def\to{\@ifnextchar^ {\t@@}{\t@@^{}}}
\def\from{\@ifnextchar^ {\t@@left}{\t@@left^{}}}
\def\tofro{\@ifnextchar^ {\t@fr@}{\t@fr@^{}}}
\def\To{\@ifnextchar^ {\T@@}{\T@@^{}}}
\def\From{\@ifnextchar^ {\T@@left}{\T@@left^{}}}
\def\Two{\@ifnextchar^ {\Tw@@}{\Tw@@^{}}}
\def\Tofro{\@ifnextchar^ {\T@fr@}{\T@fr@^{}}}

\makeatother

\input{diagxy}

\begin{document}
\maketitle

\begin{abstract}
Connections are an important tool of differential geometry.  This paper investigates their definition and structure in the abstract setting of tangent categories.  At this level of abstraction we derive several classically important results about connections, including the Bianchi identities, identities for curvature and torsion, almost complex structure, and parallel transport. 
\end{abstract}

\tableofcontents 

%%%%%%%%%%%%%%%%%%%%%%%%%%%%%%%%%%%%%%%%%%%%%%%%%%%%%%%%%%%%%%%%%%%%%%%%%

\section{Introduction}
The fundamental and most important theorem of Riemannian Geometry is that any Riemannian manifold has a unique (metric compatible) connection -- the Levi-Civita connection -- which determines its geodesics and, thus, its geometry.  A connection on a smooth manifold in a very real sense encodes ``geometry'': not only does it determine geodesics, but also curvature, parallel transport, and, thus, the holonomy of the space.  This has made the investigation of connections, in their various incarnations, a central topic of Differential Geometry.    

Tangent categories provide an abstract and very general framework for Differential Geometry by axiomatizing the behaviour of the tangent bundle functor.  While the category of smooth manifolds is a standard example of a tangent category, there are {\em many\/} other sources of important examples:  categories of schemes from algebraic geometry, microlinear spaces from Synthetic Differential Geometry,  convenient categories of manifolds, tropical geometry, models of the differential $\lambda$-calculus from computer science, and categories of combinatorial species from combinatorics.  

Since connections are important in Differential Geometry, a natural question to ask is how one can define and work with connections in an abstract tangent category.  Providing answers to this question is the main goal of this paper.  In particular, we will use \emph{differential bundles} (see \cite{diffBundles}) -- the analog of vector bundles for tangent categories -- to give several different notions of a connection in a tangent category and explore how these different notions are related.  Classically these different notions are equivalent, however, this is not always the case for general tangent categories.

Perhaps the most widely used notion of connection is that of a \emph{Koszul connection} (on a tangent bundle, this is also known as a \emph{covariant derivative}).  This involves giving an operation on the global sections of a vector bundle $q:E \to M$.  In a tangent category one cannot expect the global sections of an object to completely determine the object (already in algebraic geometry this is not the case).  However, an important result of Patterson \cite{patterson} shows that, in the classical case, a Koszul connection on a vector bundle $q:E \to M$ bijectively corresponds to certain linear maps $K:T(E) \to E$ (where $T$ is the tangent bundle functor).  This latter notion can be reformulated in any tangent category (see Definition \ref{defnVertConnection}).  We call such a connection a \emph{vertical connection} as it is a special section of the \emph{vertical lift} -- a map which every differential bundle possesses.  We then show that this definition supports a very straightforward formulation of the concepts of curvature and torsion.

A different way of defining a connection on a vector bundle $q:E \to M$ is due to Ehresmann.  It involves splitting the tangent bundle $T(E)$ into a ``horizontal'' and ``vertical'' subspace.  This idea has a natural corresponding formulation in a tangent category as a special retract of the \emph{horizontal descent}  $\<T(q),p\>: T(E) \to T(M) \x_M E$ which we call a \emph{horizontal connection} (see Definition \ref{defnHorizConnection}).   This way of defining a horizontal connection is essentially the same as that used in Synthetic Differential Geometry (see Definition 1.1 of \cite{kockConnections}), and is used in some texts in ordinary differential geometry (see, for example \cite[pg. 104]{lang}).  

In many applications (in particular, see our discussion of  parallel transport in Section \ref{secParallelTransport}), it is useful to have both a vertical and a horizontal connection.  Thus, in a general tangent category, we define a full \emph{connection} to consist of both a vertical and a horizontal connection satisfying two compatibility conditions (see Definition \ref{defnConnection}). 

Classically, the two notions, a  Koszul/vertical connection and an Ehresmann/horizontal connection are equivalent and, indeed, generate a (full) connection.  However, the proof of this fact uses two key facts about smooth manifolds: (1) that one can negate tangent vectors, and (2) that every smooth manifold always has an Ehresmann connection.  In an arbitrary tangent category one does not necessarily have negatives as to do so would rule out  key examples arising from combinatorics and Computer Science.  Furthermore, there is no reason why a general object in a tangent category should possess an Ehresmann connection.  However, when one has negation and sufficient horizontal connections we show that these notions are indeed equivalent.  In particular, given the existence of negatives, we show having a horizontal connection implies one has a connection (see Proposition \ref{horizontalToConnection}).  Furthermore, given both negatives and a horizontal connection on an object, we show that having a vertical connection implies that one has a connection (see Proposition \ref{propConstructingHorizLift}).  

In addition to establishing these various forms for connections, we show that a number of their fundamental properties generalize to tangent categories.  In particular, we show how to construct horizontal and vertical connections from existing ones by pulling back or applying the tangent functor $T$.  We also provide generalizations of various classical tools such as the Bianchi identities, the existence of parallel transport from a connection, and the existence of an almost complex structure from an affine connection.  

This is by no means the end of the story on connections in tangent categories.  Indeed, an important classical property of connections is their correspondence to sprays \cite{sprays} which we hope to explore in future work.  

The paper is organized as follows.  In section \ref{secDiffBundles}, we review the idea of a \emph{differential bundle} in a tangent category from \cite{diffBundles}.  These are the analog of vector bundles in the category of smooth manifolds.  The section also contains some new results concerning linear maps between these bundles.  

In section \ref{secVertConn}, we look at how to define the analog of the Koszul connection in this setting, calling such connections \emph{vertical connections}.  We prove a variety of basic results about this notion of connection.   In addition, we show how to define curvature and torsion for such connections, and show the equivalence of such connections to the notion of a \emph{Finsler connection}.

In section \ref{secHorizConn}, we look at how to define the analog of Ehresmann connections in this setting, calling such connections \emph{horizontal connections}.  Again, we prove a variety of basic results about these connections.

In section \ref{secConn}, we define a \emph{connection} as a vertical connection together with a horizontal connection, satisfying two compatibility axioms.  We prove some basic results about these connections.  We show how under certain circumstances one can define a connection from the data of a vertical or horizontal connection.  We finish by recovering several classical results: given an assumption of the existence of certain \emph{curve objects}, we show how to define parallel transport for such a connection, and assuming the existence of negatives, we show how to define an almost complex structure from an affine connection.  

\medskip
Finally, we would like to draw the readers attention to several notational conventions we use:
\begin{itemize}
	\item Throughout the paper, we write composition in diagrammatic order, so that $f$, followed by $g$ is written as $fg$.
	\item We use the notation $\<f,g\>$ for the unique map into a pullback induced by the maps $f,g$, and $h \times k$ as shorthand for $\<\pi_0 h, \pi_1 k\>$ for maps $h,k$ out of a pullback 
	           (where $\pi_0$ and $\pi_1$ are the two projections of the pullback).  
	\item We occasionally use infix notation for summing terms from an additive bundle; that is, we will occasionally write $f + g$ to represent $\<f,g\>+$.
	\item To save space in long equations, we frequently omit the subscripts on the tangent category natural transformations (so that, for example, instead of writing $\ell_M c_M = \ell_M$, we simply write $\ell c = \ell$).  
\end{itemize}

%%%%%%%%%%%%%%%%%%%%%%%%%%%%%%%%%%%%%%%%%%%%%%%%%%%%%%%%%%%%%%%%%%%%%%%%%%%%%%%%%

\section{Differential bundles}\label{secDiffBundles}

We assume the reader is familiar with the basic theory of tangent categories as presented in \cite{sman3}.  In this section, we review (and slightly modify) the notion of \emph{differential bundles} in tangent categories as defined in \cite{diffBundles}.  This allows us to fix notation and add some results which we shall need about linear bundle morphisms between differential bundles.  Differential bundles are the appropriate generalization of vector bundles of smooth manifolds to tangent categories, and are the objects on which we define connections.  

\subsection{Definition and examples}\label{diffBundlesDefinition}

There are two differences with the version of differential bundles as presented in \cite{diffBundles} and the version used here.  First, we require extra pullback assumptions in order to describe horizontal connections on such bundle\footnote{Thanks to Rory Lucyshyn-Wright for noticing that an earlier version of this extra pullback assumption was insufficient.}.   Secondly, in \cite{diffBundles}, $\sigma$ and $\zeta$ were used to denote the addition and zero of an additive bundle; in order to facilitate the writing of sums in infix notation, in this paper we will use $\pq$ and $\zq$ instead.  

\begin{definition}
A \textbf{differential bundle} in a tangent category consists of an additive bundle $(q: E \to M, \pq: E_2 \to E, \zq: M \to E)$ 
with a map $\lambda: E \to T(E)$, called the {\bf lift}, such that
\begin{itemize}
        \item finite wide pullbacks of $q$ along itself exist and are preserved by each $T^n$ (call these objects $E_n$);
		\item for any $n, m,k \in \mathbb{N}$, the pullback of $E_n \to M$ along $T_mM \to M$ exists and is preserved by each $T^k$;
        \item $(\lambda,0_M)$ is an additive bundle morphism from $(E, q,\pq,\zq)$ to $(T(E), T(q), T(\pq), T(\zq))$;
	\item $(\lambda,\zq)$ is an additive bundle morphism from $(E, q,\pq,\zq)$ to $(T(E), p_E,+_E,0_E)$;
	\item the {\bf universality of the lift} requires that the following is a pullback:
	       $$\xymatrix{E_2 \ar[d]_{\p_0q=\pi_1q} \ar[rrr]^{\mu := \<\p_0\lambda, \p_10\>T(\pq)} & & & T(E) \ar[d]^{T(q)} \\ M \ar[rrr]_{0} && & T(M)}$$
	\item the equation $\lambda \ell_E = \lambda T(\lambda)$ holds.
\end{itemize}
We shall write $\sf q$ to denote the entire bundle structure $(q,\pq,\zq,\lambda)$.  
\end{definition}

There are several important examples of differential bundles and methods of constructing new differential bundles from existing ones.  Proofs for these results can all be found in \cite{diffBundles}.  

\begin{example}~ \label{basicdiffbundles}
{\em 
\begin{enumerate}[(i)]
%\item Any object has an associated ``trivial'' differential bundle ${\sf 1}_M =(1_M,1_M,1_M,0_M)$.  Any differential bundle over $M$ has 
%a unique bundle map to this bundle, $(q,1_M): {\sf q} \to {\sf 1}_M$, which is the identity on the base:
%$$\xymatrix{E \ar[d]_q \ar[rr]^q & & M \ar[d]^{1_M} \\ M \ar[rr]_{1_M} && M}$$
%Furthermore this is a linear map as $\lambda T(q) = q 0$.  We shall see that ${\sf 1}_M$ is the final differential bundle in the 
%fibre over $M$.  Given any $f: N \to M$ this can also be viewed as a linear morphism between the trivial bundles $(f,f): {\sf 1}_N \to {\sf 1}_M$.
%Furthermore, there is always a linear zero bundle morphism $(\zq,1): {\sf 1}_M \to {\sf q}$.
\item In the tangent category of smooth manifolds, every vector bundle is a differential bundle.  The  converse, however,  is not true. Usually a vector bundle, $q: E \to M$, in the category of smooth manifolds 
is assumed to have a fixed vector space, $V$, such that for each $x \in M$ there is a neighbourhood, $x \in U$, over which the bundle is trivial with fibre $V$,  thus it takes the form $\pi_1: U \x V \to U$.    Differential 
bundle do not require that the fibres all be isomorphic to a particular $V$.
\item In a model of synthetic differential geometry $\X$, for any object $M$ in $\X$, Euclidean $R$-modules in $\X/M$ are the same as differential bundles over $M$. 
\item In any tangent category, for any object $M$, the tangent bundle $T(M)$ is a differential bundle, with structure
	\[ {\sf p}_M := (p_M: T(M) \to M,+_M,0_M,\ell_M) \]  
\item In any tangent category, if ${\sf q}$ is a differential bundle, then $$T({\sf q}):= (T(q),T(\pq),T(\zq),T(\lambda)c)$$ is also a differential bundle (the extra pullback requirements in this paper is automatic in this case). 
\item If ${\sf q}$ is a differential bundle over $M$ and $f: X \to M$ is a map such that the pullback of $q$ along $f$ (notated $f^*(E)$, with projections $f^*(q)$ and $f^*_E$, as below) exists and is preserved by $T$, and such that the pullback of $f^*(q)$ along $p_X$ exists and is preserved by each $T^n$, then $f^*(q)$ is a differential bundle, with lift $f^*(\lambda)$ given by
$$\xymatrix{ & T(f^{*}(E)) \ar[dd]|{\hole}^<<<<{T(f^{*}(q))} \ar[rr]^{T(f^{*}_E)} & & T(E) \ar[dd]^{T(q)} \\ 
           f^*(E) \ar[dd]_{f^{*}(q)}  \ar@{..>}[ur]^{f^{*}(\lambda)} \ar[rr]^>>>>>>>{f^{*}_E} & & E \ar[dd]^<<<<<<q \ar[ru]_{\lambda} \\
           & T(X) \ar[rr]_<<<<<<{T(f)}|>>>>>>>>>>>>{~~} & & T(M) \\
           X \ar[ur]^{0} \ar[rr]_{f} & & M \ar[ur]_{0}}$$
This bundle will be notated as $f^{*}({\sf q})$.  
\item An example of the above is when one has two bundles ${\sf q,q'}$ both over $M$.  In this case, one can take the pullback of $q \times q'$ along the diagonal $\Delta: M \to M \times M$, giving a bundle on $E \times_M E'$ over $M$.  We will write ${\sf q} \times {\sf q'}$ for this bundle: it is often called the Whitney sum.
\item In a Cartesian tangent category, differential bundles over a terminal object $1$ are also known as \emph{differential objects}; see Section 3 of \cite{diffBundles} for a discussion of this idea.  In particular, differential objects have in addition to their lift $\lambda: E \to T(E)$ a map $\hat{p}: T(E) \to E$, known as the \emph{principal projection}.  In the category of smooth manifolds, differential objects are vector spaces, and in models of SDG, differential objects are Euclidean $R$-modules.  
 \end{enumerate} }
\end{example}

Recall from \cite[Lemma 2.10]{diffBundles} that the universality of the lift in the definition of a differential bundle may alternatively be expressed by saying that the diagram 
$$E_2 \to^{\mu} T(E) \Two^{T(q)}_{pq0} T(M)$$
is an equalizer.  In particular, this means that for any $f: X \to T(E)$ with $fT(q) = fpq0$, there is a map from $X$ to $E_2$ which we shall write as $f_{|\mu}$.  By post-composing this map with $\pi_0: E_2 \to E$ we get a map from $X$ to $E$, which we write as $\{f\}$.  This operation and its properties (see \cite[Lemma 2.10, 2.12, 2.13]{diffBundles}) will be used a few times in this paper.

\subsection{Linear bundle morphisms}

In this section we recall the notion of \emph{linear} morphisms between differential bundles and prove some new results about these morphisms.  

\begin{definition}
Let $\sf q$ and $\sf q'$ be differential bundles.  A \textbf{bundle morphism} between these bundles simply consists of a pair of maps $f_1: E \to E'$, $f_0: M \to M'$ such that $f_1q' = qf_0$ 
(first diagram below).  A bundle morphism is \textbf{linear} in case, in addition, it preserves the lift, that is $f_1\lambda' = \lambda T(f_1)$ (the 
second diagram below).
$$\xymatrix{E \ar[d]_q \ar[rr]^{f_1} && E' \ar[d]^{q'} \\ M \ar[rr]_{f_0} && M'} ~~~~~~ 
 \xymatrix{E \ar[d]_\lambda \ar[rr]^{f_1} & & E' \ar[d]^{\lambda'} \\ T(E) \ar[rr]_{T(f_1)} && T(E')}$$
\end{definition}

%Clearly the differential bundles of a tangent category $\X$, with their morphisms, form a category: we write this as ${\sf Bun}(\X)$.  
%Furthermore, there is an obvious functor: 
%$$P: {\sf Bun}(\X) \to \X: \begin{array}[c]{c} \xymatrix{{\sf q} = (q,\pq,\zq,\lambda) \ar[d]^{(f,g)} \\ {\sf q}' =(q',\pq',\zq',\lambda')} \end{array} \mapsto 
 %                          \begin{array}[c]{c} \xymatrix{M \ar[d]^{g} \\ M'} \end{array}$$
%In \cite{diffBundles}, we showed that this was a fibration.  We also shows that the linear morphism carve 
%out a subcategory ${\sf LBun}(\X) \subseteq {\sf Bun}(\X)$ which is a subfibration of $P$.

By Proposition 2.16 from \cite{diffBundles}, linear bundle morphisms are automatically \emph{additive}, in the sense that they preserve the addition and zero of $\sf q$.  

The following result shows that a great variety of morphisms in a tangent category are linear.    

\begin{proposition}\label{examplesLinearMorphisms}
Let $\sf q$ be a differential bundle on $E$ over $M$.  
\begin{enumerate}[(i)]
	\item The identity bundle morphism $(1_E,1_M)$ is linear. 
	\item $(p_E,p_M)$ is a linear bundle morphism from $T({\sf q})$ to ${\sf q}$.
	\item $(0_E,0_M)$ is a linear bundle morphism from ${\sf q}$ to $T({\sf q})$.
	\item $(\pq,1)$ is a linear bundle morphism from ${\sf q} \times {\sf q}$ to ${\sf q}$.
	\item $(\lambda,0)$ is a linear bundle morphism from ${\sf q}$ to $T({\sf q})$.
	\item $(\lambda,\zq)$ is a linear bundle morphism from ${\sf q}$ to $\sf p_E$.
	\item For any $X$, $(c_X,1)$ is a linear bundle morphism from $T({\sf p_X})$ to ${\sf p_{T(X)}}$ and from ${\sf p_{T(X)}}$ to $T({\sf p_X})$.  
	\item $(c_E,c_M)$ is a linear bundle morphism from $T(T({\sf q}))$ to itself.  
	\item For any $f: X \to Y$, $(T(f),f)$ is a linear bundle morphism from ${\sf p_X}$ to ${\sf p_Y}$.
	\item For any $f: X \to M$, $(f^*_E,f)$ is a linear bundle morphism from $f^*({\sf q})$ to $\sf q$.  
\end{enumerate}
\end{proposition}
\begin{proof}
\begin{enumerate}[(i)]
	\item Straightforward.
	\item The bundle morphism requirement is naturality of $p$: $pq = T(q)p$.  For linearity,
		\[ T(\lambda)cT(p) = T(\lambda)p = p \lambda \]
	by properties of $p$ and $c$.
	\item The bundle morphism requirement is naturality of $0$: $0T(q) = q0$.  For linearity,
		\[ 0T(\lambda)c = \lambda 0 c = \lambda T(0) \]
	by properties of $0$ and $c$.
	\item The bundle morphism requirement is one of the axioms for an additive bundle: $\pq q = \pi_0q = \pi_1 q$.  The linearity is additivity of $\lambda$ from ${\sf q}$ to $T({\sf q})$: $\pq \lambda = (\lambda \times \lambda)T(\pq)$.

	\item That $(\lambda,0)$ is a bundle morphism is one of the axioms of a differential bundle.  For linearity, we have
		\[ \lambda T(\lambda)c = \lambda \ell c = \lambda \ell = \lambda T(\lambda) \]
	using another axiom of a differential bundle ($\lambda \ell = \lambda T(\lambda)$).
	\item That $(\lambda,\zq)$ is a bundle morphism is one of the axioms of a differential bundle.  The linearity axiom is precisely one of the other axioms ($\lambda \ell = \lambda T(\lambda)$).
	\item The bundle morphism requirements are part of the axioms of a tangent category.  The linearity requirements ($c\ell = T(\ell)cT(c)$ and $c T(\ell)c = \ell T(c)$) are simply re-arrangements of the tangent category axiom $\ell T(c)c = c T(\ell)$, using $c^2 = 1$.
	\item The pair is a bundle morphism by naturality of $c$: $cT^2(q) = T^2(q)c$.  For linearity,
		\[ cT(T(\lambda)c)c = cT^2(\lambda)T(c)c = T^2(\lambda)cT(c)c = T^2(\lambda)T(c)cT(c) = T(T(\lambda) c)cT(c) \]
	by naturality and coherence of $c$.  
	\item The bundle morphism requirement is naturality of $p$: $T(f)p = pf$.  Linearity is naturality of $\ell$: $T(f)\ell = \ell T^2(f) = \ell T(T(f))$.
	\item The bundle morphism and linearity requirements follow directly from the definition of the bundle $f^*({\sf q})$.  
\end{enumerate}
\end{proof}

In addition, there are many ways to construct new linear bundle morphisms from existing ones.  

\begin{proposition}\label{constructingLinearMorphisms}
Let $\sf q,q',q''$ be differential bundles.  
\begin{enumerate}[(i)]
	\item If $f: {\sf q} \to {\sf q'}$ and $g: {\sf q'} \to {\sf q''}$ are linear bundle morphisms then so is the composite $fg$.
	\item If $(f_1,f_0): {\sf q} \to {\sf q'}$ is linear, then so is $(T(f_1),T(f_0)): T({\sf q}) \to T({\sf q'})$.
	\item Suppose that $f: {\sf q} \to {\sf q'}$ is a linear bundle morphism and a retract of some bundle morphism $s: {\sf q'} \to {\sf q}$.  If $g: {\sf q'} \to {\sf q''}$ is any bundle morphism such that $fg$ is linear, then $g$ is linear.  
	\item If $f: {\sf q} \to {\sf q'}$ is a linear bundle morphism with a bundle morphism inverse $f^{-1}: {\sf q'} \to {\sf q}$, then $f^{-1}$ is linear.  
	\item If $q'$ and $q''$ are both over $M'$ and $q' \times q''$ exists, with $(f_1,f_0)$ a linear bundle morphism from $\sf q$ to $\sf q'$ and $(g_1,g_0)$ a linear bundle morphism from $\sf q$ to $\sf q''$ such that $f_1q' = g_1q''$, then $(\<f_1,g_1\>,f_0)$ and $(\<f_1,g_1\>,g_0)$ are linear bundle morphisms from $\sf q$ to ${\sf q'} \times {\sf q''}$. 
	\item If $f = (f_1,f_0)$ and $g = (g_1,g_0)$ are linear bundle morphisms from $\sf q$ to $\sf q'$ such that $f_1q' = g_1q'$, then $(\<f_1,g_1\>\pq,f_0)$ is also a linear bundle morphism from $\sf q$ to $\sf q'$.
	\item If $f = (f_1,f_0)$ and $g = (g_1,g_0)$ are linear bundle morphisms from $\sf q$ to $T({\sf q'})$ such that $f_1p = g_1p$ and $f_0p = g_0p$, then $f+g := (f_1 + g_1,f_0 + g_0) = (\<f_1,g_1\>+, \<f_0,g_0\>+)$ is also a linear bundle morphism from $\sf q$ to $T({\sf q'})$. 
\end{enumerate}
\end{proposition}
\begin{proof}
\begin{enumerate}[(i)]
	\item Straightforward.  
	\item The bundle requirement is functoriality of $T$: $T(f_1)T(q) = T(f_1q) = T(qf_0) = T(q)T(f_0)$.  Linearity uses functoriality of $T$ and naturality of $c$:
		\[ T(f_1)T(\lambda')c = T(f_1\lambda)c = T(\lambda T(f_1))c = T(\lambda)T^2(f_1)c = T(\lambda)c T^2(f_1). \]
	\item This can be found in \cite{diffBundles}, Proposition 2.18.i.  
	\item Follows immediately from (iii).  
	\item The bundle requirement for the first is $\<f_1,g_1\>\pi_0 q' = f_1q' = qf_0$, while the second is similar since the projection for ${\sf q'} \times {\sf q''}$ is also $\pi_1 q''$, so $\<f_1,g_1\>\pi_1 q'' = g_1q' = qg_0$.  The linearity requirement is the same for both, and similarly straightforward:
		\[ \<f_1,g_1\>(\lambda' \times \lambda'') = \<f_1 \lambda', g_1 \lambda''\> = \<\lambda T(f_1), \lambda T(f_2)\> = \lambda T(\<f_1,f_2\>). \]
	\item $(\<f_1,g_1\>\pq,f_0)$ is the composite of $(\<f_1,g_1\>,f_0)$ and $(\pq,1)$, which are linear by the result above and by Lemma \ref{examplesLinearMorphisms}.iv.  
	\item For the bundle morphism requirement, using naturality of $+$,
		\[ \<f_1,g_1\>+T(q') = \<f_1T(q'),g_1T(q')\>+ = \<qf_0,qg_0\>+ = q\<f_0,g_0\>+. \]
	For linearity, using naturality of $+$ and additivity of $c$,
	\begin{eqnarray*}
	&   & \<f_1,g_1\>+T(\lambda)c \\
	& = & \<f_1T(\lambda),g_1T(\lambda)\>+c \\
	& = & \<f_1T(\lambda)c,g_1T(\lambda)c\>T(+) \\
	& = & \<\lambda f_1, \lambda g_1\>T(+) \\
	& = & \lambda T(\<f_1,g_1\>+)
	\end{eqnarray*}
\end{enumerate}
\end{proof}

The bracketing operation associated to differential bundles (see the discussion at the end of section \ref{diffBundlesDefinition}) preserves linearity in two different ways.  

\begin{proposition}\label{propBracketLinear1}(Bracketing of linear bundle morphisms {\rm I})
Suppose $\sf q'$ and $\sf q$ are differential bundles, with $(f_1,f_0): {\sf q'} \to T({\sf q})$ a linear bundle morphism such that $f_1T(q) = f_1T(q)p0$.  Then $(\{f_1\},f_0p)$ is a linear bundle morphism from $\sf{q'}$ to $\sf q$.  
\end{proposition}
\begin{proof}
It is a bundle morphism since by \cite[Lemma 2.12.iii]{diffBundles}, $\{f_1\}q = f_1T(q)p = q'(f_0p)$.  For linearity, consider
%want the following diagram to commute:
%\[
%\bfig
%	\square<500,250>[E'`E`T(E)'`T(E);\{f\}`\lambda'`\lambda`T(\{f\})]
%\efig
%\]
\begin{eqnarray*}
&   & \lambda'T(\{f_1\}) \\
& = & \lambda'\{T(f_1)c\} \mbox{ (by \cite[Lemma 2.13]{diffBundles})} \\
& = & \{\lambda'T(f_1)c\} \mbox{ (by \cite[Lemma 2.12.i]{diffBundles})} \\
& = & \{f_1 T(\lambda)cc\} \mbox{ ($f$ is linear by assumption)} \\
& = & \{f_1 T(\lambda)\} \\
& = & \{f_1\}\lambda \mbox{ (by \cite[Lemma 2.12.ii]{diffBundles})}
\end{eqnarray*}
as required.  
\end{proof}

A similar result holds for the other differential structure on $T(E)$:

\begin{proposition}\label{propBracketLinear2}(Bracketing of linear bundle morphisms {\rm II})
Suppose $\sf q'$ and $\sf q$ are differential bundles, with $(f_1,f_0): {\sf q'} \to {\sf p_E}$ a linear bundle morphism such that $f_1T(q) = f_1T(q)p0$.  Then $(\{f_1\},f_0q)$ is a linear bundle morphism from $\sf{q'}$ to $\sf q$.  
\end{proposition}
\begin{proof}
The pair $(\{f_1\},f_0q)$ is a bundle morphism since by \cite[Lemma 2.12.iii]{diffBundles},
	\[ \{f_1\}q = f_1T(q)p = f_1pq = q'(f_0q). \]
For linearity, by \cite[Lemma 2.13]{diffBundles} and the fact that $f_1$ is linear,
%\[
%\bfig
%	\square<500,250>[E'`E`T(E')`T(E);\{f\}`\lambda'`\lambda`T(\{f\})]
%\efig
%\]
	\[ \lambda'T(\{f_1\}) = \lambda'\{T(f_1)c\} = \{\lambda'T(f_1)c\} = \{f_1 \ell c\} = \{f_1 \ell \} \]
To complete the proof of linearity, we thus wish to show that $\{f_1 \ell \}$ equals $\{f_1\}\lambda$.  Thus, it suffices to show that $\{f_1\}\lambda$ has the same universal property (\cite[Lemma 2.10.ii]{diffBundles}) as $\{f_1 \ell \}$ (Note that the bracketing of this last term is in relation to the lift $T(\lambda)c$).  Thus, we consider
\begin{eqnarray*}
&   & \<\{f_1\}\lambda T(\lambda)c,f_1\ell p0\>T^2(\pq) \\
& = & \<\{f_1\}\lambda \ell c,f_1p0\ell\>T^2(\pq) \mbox{ (coherence for $\lambda$ and naturality of $p0$)} \\
& = & \<\{f_1\}\lambda \ell ,f_1p0\ell\>T^2(\pq) \\
& = & \<\{f_1\}\lambda,f_1p0\>T(\pq)\ell \mbox{ (naturality of $\ell$)} \\
& = & f_1\ell
\end{eqnarray*}
Thus $\{f\}\lambda$ has the same universal property as $\lambda'T(\{f\}) = \{f \ell_E \}$, so $\{f\}\lambda = \lambda'T(\{f\})$, as required for linearity.   
\end{proof}

%%%%%%%%%%%%%%%%%%%%%%%%%%%%%%%%%%%%%%%%%%%%%%%%%%%%%%%%%%%%%%%%%%
%%%%%%%%%%%%%%%%%%%%%%%%%%%%%%%%%%%%%%%%%%%%%%%%%%%%%%%%%%%%%%%%%%

\section{Vertical connections}\label{secVertConn}

The first notion of connection on a differential bundle $\sf q$ we consider consists of a map $K: T(E) \to E$ which is a retract of the lift $\lambda: E \to T(E)$ and is appropriately linear.  This way of describing a connection seems to have first appeared in \cite{dombrowski}, with a detailed comparison of its relationship to the covariant derivative in \cite{patterson}.

%%%%%%%%%%%%%%%%%%%%%%%%%%%%%%%%%%%%%%%%%%%%%%%%%%%%%%%%%%%%%%%%%%

\subsection{Definition and examples}

\begin{definition}\label{defnVertConnection}
Let $\sf q$ be a differential bundle on $E$ over $M$.  A \textbf{vertical descent} on $\sf q$ is a map $K: T(E) \to E$ which is a retract of $\lambda: E \to T(E)$.  A \textbf{vertical connection} on $\sf q$ is a vertical descent $K$ such that
\begin{enumerate}[{\bf [C.1]}]
	\item $(K,p):T({\sf q}) \to {\sf q}$ is a linear bundle morphism;
	\item $(K,q): {\sf p}_E \to {\sf q}$ is a linear bundle morphism.  
\end{enumerate}
A vertical connection is said to be \textbf{affine} when $\sf q$ is a tangent bundle ${\sf p}_M$ (so that $K: T^2(M) \to T(M)$).
\end{definition}

The definition in this form needs unwrapping.  The retract requirement is $\lambda K = 1_E$.  To ensure that the two bundle morphisms are linear morphisms requires that they are minimally bundle morphisms, that is $K q = T(q) p  = p q$.  To ensure they are linear requires that $\ell T(K) = K \lambda$ and, recalling that 
the lift of $T({\sf q})$ is $T(\lambda)c:T(E) \to T^2(E)$, that $T(\lambda)cT(K) = K \lambda$.  We collect these equations into the following:

\begin{lemma}\label{vertDescConditions}  A map $K: T(E)  \to E$ is a vertical connection for the bundle ${\sf q}$ if and only if the following equations are  satisfied:
\begin{enumerate}[(a)] 
\item $\lambda K= 1_E$;
\item $K q = p q$;
\item $K \lambda = \ell T(K)$;
\item $K \lambda = T(\lambda)c T(K)$.
\end{enumerate}
\end{lemma}

A useful result will also be the following.
\begin{lemma}\label{lemmaVConnector}
If $K$ is a vertical connection on ${\sf q}$ then $\mu K = \pi_0$.
\end{lemma}
\begin{proof}
Using the fact that $K$ is additive, we have
\begin{eqnarray*} 
\mu K & = & \<\pi_0 \lambda, \pi_1 0\>T(\pq)K = \< \pi_0 \lambda K, \pi_1 0 K\> \pq \\
& =  & \< \pi_0 \lambda K, \pi_1 q \zq \> \pq = \< \pi_0, \pi_0 q \zq \> \pq = \pi_0. 
\end{eqnarray*}
\end{proof}

Classically, every Cartesian space $\mathbb{R}^n$ has a canonical choice of connection (as a vector bundle over $1$); a similar result holds in tangent categories for differential objects.  We follow the notation for differential objects used in \cite[Section 3.1]{diffBundles}.  

\begin{proposition}\label{diffObjectVD}
In a Cartesian tangent category, if $A$ is a differential object (that is, a differential bundle over $1$) then $\hat{p}: T(A) \to A$, the principal projection, is a vertical connection on this differential bundle.
\end{proposition}
\begin{proof}
By \cite[Proposition 3.4]{diffBundles}, $\lambda \hat{p} = 1$ and $\hat{p}$ is an additive bundle morphism for both bundle structures on $T(A)$.  All that remains is to show it is a linear bundle morphism for these structures.  

For the first linearity, we need $\ell T(\hat{p}) = \hat{p}\lambda$.  These are both maps into $T(A)$, which is a product with projections $\hat{p}$ and $p$, so it suffices to show this equality holds when post-composing by these projections.  Indeed, again using \cite{diffBundles}, Proposition 3.4,
	\[ \ell T(\hat{p})\hat{p} = \hat{p} = \hat{p} \lambda \hat{p} \]
while using axioms of a tangent category,
	\[ \ell T(\hat{p})p = \ell p \hat{p} = p0\hat{p} = p ! \zq = ! \zq \]
and
	\[ \hat{p} \lambda p = \hat{p} ! \zq = ! \zq \]

For the second linearity, we need $T(\lambda)c T(\hat{p}) = \hat{p} \lambda$.  Again, we can check this by checking the equality when post-composing by the projections of $T(A)$: using Proposition 3.4 and 3.6 of \cite{diffBundles},
	\[ T(\lambda)c T(\hat{p})\hat{p} = T(\lambda)T(\hat{p})\hat{p} = T(\lambda \hat{p})\hat{p} = \hat{p} = \hat{p} \lambda \hat{p} \]
while
	\[ T(\lambda)c T(\hat{p})p = T(\lambda)cp\hat{p} = T(\lambda)T(p) \hat{p} =  T(\lambda p) \hat{p} = T(! \zq)\hat{p} = ! T(\zq) \hat{p} = ! \zq \]
which equals $\hat{p} \lambda p$ by the calculation above.
\end{proof}
We shall refer to this as the \textbf{canonical vertical connection} on $A$.  In Corollary \ref{diffObjectVDUnique} we shall see that this vertical connection is unique.  

%\begin{theorem}(Patterson, Theorem 1)
%To give a vertical decent on a vector bundle in the category of smooth manifolds is equivalent to giving a covariant derivative on the vector bundle.
%\end{theorem}

As mentioned in the introduction, Patterson \cite[Theorem 1]{patterson} shows that in the category of smooth manifolds, giving a vertical connection on a vector bundle is equivalent to giving a (Koszul) connection \cite[pg. 227]{spivak2} on the vector bundle.  Thus, any connection in the standard sense is a vertical connection, so there is an abundance of connections in the classical case, the tangent category of smooth manifolds.  However, describing a connection in the form given above is probably unfamiliar to a general reader, and so in the next two examples we explicitly investigate the form of vertical connections on some standard smooth manifolds.    

\begin{example}\label{exampleVD}{\em ~
\begin{enumerate}[(1)] 
\item In the category of smooth manifolds, let us explicitly consider the form of an affine vertical connection on a Euclidean space $\mathbb{R}^n$:   it is a map 
$$K: T^2(\R^n)=\R^{4n} \to T(\R^n)=\R^{2n}; (w,v,y,x) \mapsto (k(w,v,y)_x, x)$$
Since $K$ is a section of $\ell$ we have $k(w,0,0)_x = w$.   We also know that $K$ is additive in two different ways.  Thus, we know $k(w_1+w_2,v_1+v_2,y)_x = k(w_1,v_1,y)_x+k(w_2,v_2,y)_x$ and $k(w_1+w_2,v,y_1+y_2)_x = k(w_1,v,y_1)_x+k(w_2,v,y_2)_x$.  This allows the following  reductions:
\begin{eqnarray*}
k(w,v,y)_x & = & k(w,0,y)_x + k(0,v,y)_x = k(w,0,0)_x + k(0,0,y)_x + k(0,v,y)_x \\ & = & w + k(0,0,1)_x  \cdot y+ k(0,1,1)_x  \cdot v \cdot y\\
k(w,v,y)_x & = & k(w,v,0)_x + k(0,v,y)_x = k(w,0,0)_x + k(0,v,0)_x + k(0,v,y)_x \\ & = & w + k(0,1,0)_x \cdot v + k(0,1,1)_x  \cdot v \cdot  y
\end{eqnarray*}
These imply that $k(0,0,1)_x \cdot y = k(0,1,0) \cdot v$; setting $y=v$ shows $k(0,0,1)_x = k(0,1,0)_x$, and setting $y=0$ shows that both functions are identically zero.  Thus, setting $\psi(x) := k(0,1,1)_x$ we have shown so far:
$$k(w,v,y)_x = w + \psi(x) \cdot v \cdot y$$
where minimally $\psi$ is a smooth function from $\R^n$ to bilinear forms on $\R^n$: recall that a bilinear form on $\R^n$ is just an $n \x n \x n$-matrix.  

There are two further equations to check, namely  $K \lambda = \ell T(K)$ and $K \lambda = T(\lambda)c T(K)$:
\begin{eqnarray*}
\lambda(K(w,v,y,x)) & = & \lambda(w+\psi(x) \cdot v \cdot y,x) = (w+ \psi(x) \cdot v \cdot y,0,0,x) \\
T(K)(\ell(w,v,y,x)) & = & T(K)(w,v,0,0,0,0,y,x) \\
& = & ((w'+ \psi(x) \cdot v'\cdot y + \psi(X) \cdot v \cdot y' + \left(\frac{\partial \psi(x)}{\partial x}(x) \cdot x'\right)\cdot v\cdot y)\\ & & ~~~~~[w/w',v/v',0/y',0/x',0/w,0/v,y/y,x/x], 0,0,x) \\
& = & (w+\psi(x) \cdot v\cdot y,0,0,x) \\
T(K)(c(T(\lambda)(w,v,y,x))) & = & T(K)(c(w,0,0,v,y,0,0,x)) = T(K)(w,0,y,0,0,v,0,x) \\
& = & ((w'+\psi(x) \cdot v'\cdot  y + \psi(X) \cdot v \cdot  y' + \left(\frac{\partial \psi(x)}{\partial x}(x) \cdot x'\right)\cdot v \cdot y)\\ & & ~~~~~[w/w',0/v',y/y',0/x',0/w,v/v,0/y,x/x], 0,0,x) \\
& = & (w+\psi(x) \cdot v \cdot y,0,0,x)
\end{eqnarray*}
So the general form of a vertical connection on $\R^n$ is:
$$K:T^2(\R^n) \to T(\R^n); (w,v,y,x) \mapsto (w + \psi(x) \cdot v \cdot  y,x)$$
Where $\psi(x)$ is a smooth map from $\R^n$ into $n \x n \x n$-matrices.   The values $\psi(x)_{(i,j,k)}$ of the matrix are the so called {\em Christoffel symbols\/} (of the second kind) \cite[pg. 186]{spivak2} of the connection $K$: they are often written $\Gamma^i_{jk}$.   

Every $\mathbb{R}^n$ has a canonical affine vertical connection as it is a differential object (see example \ref{exampleTrivialDiffBundleVD}).  The canonical affine vertical connection on $\R^n$ is when $\psi(x) = 0$ so is
$$K_0:T^2(\R^n) \to T(\R^n); (w,v,y,x) \mapsto (w,x).$$

\item The $n$-dimensional sphere in smooth manifolds is defined by 
$$S_n = \{ x | x \cdot x^T -1 = 0 \}$$
(where $x^T$ is the transpose of the vector $x$).  The tangent space of $S_n$ is obtained by applying the tangent functor to the equation and taking its kernel:
$$T(S_n) = \{  (y,x) | x \cdot x^T -1 = 0 , x \cdot y^T + y \cdot x^T =0 \}$$
The second equation can clearly be simplified to $y \cdot x^T = 0$, which says that the tangent space at the unit vector $x$ is orthogonal to $x$, as expected.  To obtain the second tangent space we differentiate again:
$$T^2(S_n) =\{ (w,v,y,x) | x \cdot x^T -1 = 0 , y \cdot x^T=0, v \cdot x^T = 0, v \cdot y^T + w \cdot x^T = 0 \}$$
The vertical lift has $\ell(y,x) = (y,0,0,x)$ as usual.

The canonical vertical connection for the sphere is then defined by 
$$K: T^2(S_n) \to T(S_n); (w,v,y,x) \mapsto (w + (v \cdot y^T) x, x)$$
where note that $(w + (v \cdot y^T) x) \cdot x^T = w \cdot x^T + (v \cdot y^T) (x \cdot x^T) = w \cdot x^T + v \cdot y^T = 0$.  Intuitively $w$ is orthogonally projected onto the tangent space.  This is clearly a section of the lift and a bundle morphism.  We must check the identities involving the lift:
\begin{eqnarray*}
\lambda(K(w,v,y,x)) & = & \lambda( w + (v \cdot y^T) x, x) = (w + (v \cdot y^T) x, 0,0,x) \\
T(K)(\ell(w,v,y,x)) & = & T(K)(w,v,0,0,0,0,y,x) \\
& = & ((w'+ (v' \cdot y^T)x + (v \cdot y'^T)x + (v.y^T)x')\\ & & ~~~~~[w/w',v/v',0/y',0/x',0/w,0/v,y/y,x/x],0,0,x) \\
& = & (w + (v \cdot y^T) x, 0,0,x) \\
T(K)(c(T(\lambda)(w,v,y,x))) & = & T(K)((w,0,y,0,0,v,0,x) \\
& = & ((w'+ (v' \cdot y^T)x + (v \cdot y'^T)x + (v.y^T)x')\\ & & ~~~~~[w/w',0/v',y/y',0/x',0/w,v/v,0/y,x/x],0,0,x) \\
& = & (w + (v \cdot y^T) x, 0,0,x)
\end{eqnarray*}
as required.  
%The Levi-Civita connection is clearly torsion free.
\end{enumerate} }
\end{example}

In the next section, we shall see ways to construct new vertical connections out of existing ones, increasing our supply of examples.  

%%%%%%%%%%%%%%%%%%%%%%%%%%%%%%%%%%%%%%%%%%%%%%%%%%%%

\subsection{Basic properties of vertical connections}

We begin by investigating three fundamental ways to construct vertical connections out of existing ones: by applying $T$, by pullback, and by retract.  

\begin{proposition}\label{tVertDescent}
Suppose that $K$ is a vertical connection on a differential bundle ${\sf q}$.  Then $cT(K)$ is a vertical connection on the differential bundle $T({\sf q})$.
\end{proposition}
\begin{proof}
The lift on $T({\sf q})$ is $T(\lambda)c$.    $cT(K)$ is then a retraction of that map since
	\[ T(\lambda)ccT(K) = T(\lambda)T(K) = T(\lambda K) = T(1) = 1. \]
For the first linearity condition, $(cT(K),p)$ is linear from $T^2({\sf q})$ to $T({\sf q})$ since it is the composite of the bundle morphisms
	\[ T^2({\sf q}) \to^{(c_E,c_M)} T^2({\sf q}) \to^{T(K),T(p)} T({\sf q}) \]
which are both linear by Proposition \ref{examplesLinearMorphisms}.  For the second linearity condition, $(cT(K),T(q))$ is linear from $\sf p_{T(E)}$ to $T({\sf q})$ since it is the composite of the bundle morphisms
	\[ {\sf p_{T(E)}} \to^{(c,1)} T({\sf p_E}) \to^{(T(K),T(q))} T({\sf q}), \]
which are again both linear by Proposition \ref{examplesLinearMorphisms}.  

%The other linearity condition we check explicitly.  For this we need $cT(K)T(\lambda)c = T(T(\lambda)c)cT(cT(K))$.  Indeed:
%\begin{eqnarray*}
%&   & cT(K)T(\lambda)c \\
%& = & cT(K\lambda)c \\
%& = & cT (T(\lambda)cT(K))c \mbox{ (by linearity of $K$)} \\
%& = & cT^2(\lambda)T(c)T^2(K)c \\
%& = & T^2(\lambda)cT(c)cT^2(K) \mbox{ (naturality of $c$)} \\
%& = & T^2(\lambda)T(c)cT(c)T^2(K) \mbox{ (a coherence of $c$)} \\
%& = & T(T(\lambda)c)cT(cT(K)) \\
%\end{eqnarray*}
%as required.  
\end{proof}

\begin{proposition}\label{pullbackVD}
Suppose that $K$ is a vertical connection on ${\sf q}$, a differential bundle over $M$.  Then for any map $f: X \to M$, $f^*K$ is a vertical connection on $f^*({\sf q})$, where $f^*K$ is defined by
		$$\xymatrix{ 
& T(f^{*}(E)) \ar[dd]|{\hole}^<<<<{T(f^{*}(q))} \ar[rr]^{T(f^{*}_E)} & & T(E) \ar[dd]^{T(q)} \ar[dl]_{K} \\ 
  f^*(E) \ar[dd]_{f^{*}(q)}  \ar@{<..}[ur]^{f^{*}(K)} \ar[rr]^>>>>>>>{f^{*}_E} & & E \ar[dd]^<<<<<<q  \\
  & T(X) \ar[dl]_{p} \ar[rr]_<<<<<<{T(f)}|>>>>>>>>>>>>{~~} & & T(M) \ar[dl]_{p} \\
  X  \ar[rr]_{f} & & M }$$
\end{proposition}

\begin{proof}
First, note that $f^*K$ is well-defined since by definition of $K$ and naturality of $p$,
	\[ T(f^*_E)Kq = T(f^*_E)pq = pf^*_Eq = pf^*(q)f = T(f^*(q))pf. \]
To show that $f^*K$ is a vertical connection, we will explicitly check the four conditions from Lemma \ref{vertDescConditions}.  For the first,
	\[ f^*(\lambda)f^*(K) = \<f^*_E\lambda, f^*(q)0\>\<T(f^*_E)K,T(f^*(q))p\> = \<f^*_E\lambda K, f^*(q)0p\> = \<f^*_E,f^*(q)\> = 1_{f^*E}. \]
For the second condition, by definition of $f^*(K)$,
	\[ f^*(K)f^*(q) = T(f^*(q))p = pf^*(q). \]
For the third and fourth conditions, we first calculate 
	\[ f^*(K)f^*(\lambda) = \<T(f^*_E)K,T(f^*(q))p\>\<f^*_E\lambda, f^*(q)0\> = \<T(f^*_E)K\lambda,T(f^*(q))p0\> (\dagger) \]
Then we have the third condition since
\begin{eqnarray*}
&   & \ell T(f^*(K)) \\
& = & \ell \<T^2(f^*_E)T(K),T^2(f^*(q))T(p)\> \\
& = & \<\ell T^2(f^*_E)T(K),\ell T^2(f^*(q))T(p)\> \\
& = & \<T(f^*_E)\ell T(K),T(f^*(q))\ell T(p) \> \mbox{ (by naturality of $\ell$)} \\
& = & \<T(f^*_E) K \lambda,T(f^*(q)) p0 \> \mbox{ (by condition on $K$ and coherence of $\ell$)} \\
& = & f^*(K)f^*(\lambda) \mbox{ (by $\dagger$)} \\
\end{eqnarray*}
Finally, for the fourth condition,
\begin{eqnarray*}
&   & T(f^*(\lambda))cT(f^*(K)) \\
& = & \<T(f^*_E)T(\lambda),T(f^*(q))T(0)\> c \<T^2(f^*_E)T(K),T^2(f^*(q))T(p) \> \\
& = & \<T(f^*_E)T(\lambda),T(f^*(q))T(0)\> \<T^2(f^*_E)c T(K),T^2(f^*(q))c T(p) \> \mbox{ (by naturality of $c$)} \\
& = & \<T(f^*_E)T(\lambda)c T(K),T(f^*(q))T(0)c T(p) \> \mbox{ ($T^2(f^*_E),T^2(f^*(q))$ are the projections of $T^2(f^*E)$)} \\
& = & \<T(f^*_E) K \lambda,T(f^*(q))T(0)p \> \mbox{ (condition on $K$ and coherence of $c$)} \\
& = & \<T(f^*_E) K \lambda,T(f^*(q))p0 \> \mbox{ (naturality of $p$)} \\
& = & f^*(K)f^*(\lambda) \mbox{ (by $\dagger$)} \\
\end{eqnarray*}
\end{proof}

\begin{example}\label{exampleTrivialDiffBundleVD}{\em  ~
By proposition \ref{diffObjectVD}, any differential object $A$ has a canonical vertical connection $\hat{p}: T(A) \to A$.  Thus for any $M$, whenever the product $A \times M$ exists, the following square is a pullback:
\[
\bfig
	\square<500,250>[A \times M`A`M`1;\pi_0`\pi_1`!`!] 
\efig
\]
and hence by the above result the bundle $\pi_1: A \times M \to M$ has a vertical connection given by pulling back the vertical connection $\hat{p}$ on $A$.   In particular:
\begin{itemize}
	\item Taking $M = A$ itself, we have $T(A) \cong A \times A$ with projections $\hat{p}$ and $p$, and so $p: T(A) \to A$ (ie., the differential bundle $\sf p_A$) acquires a connection, which we shall refer to as the \textbf{canonical affine vertical connection} on $A$.   We can explicitly calculate that the vertical connection in this case is 
		\[ T^2(A) \to^{\<T(\hat{p})\hat{p},pp\>} T(A) \]
	Note that this is typically not the only connection on $\sf p_A$: see Example \ref{exampleVD}(1).  
	\item In the category of smooth manifolds, any open subset $U \subseteq R^n$ has the property that $T(U) \cong U \times R^n$, so the tangent bundle of any open subset has a canonical vertical connection, with a formula similar to the previous example. 
\end{itemize}
}
\end{example}

\begin{proposition}\label{retractVD}
Suppose that we have a section/retraction pair, $\rhd^s_r: {\sf q} \to {\sf q'}$, in the category of differential bundles and linear maps:
\[
\bfig
	\morphism(0,0)|a|/{@{>}@/^1em/}/<500,0>[{\sf q'}`{\sf q};(s_1,s_0)]
	\morphism(500,0)|b|/{@{>}@/^1em/}/<-500,0>[{\sf q}`{\sf q'};(r_1,r_0)]
\efig
\]
Then if $K$ is a vertical connection on $\sf q$, then $T(s_1)Kr_1$ is a vertical connection on $\sf q'$.
\end{proposition}
\begin{proof}
For the retract condition, we use the assumption that $s$ is linear:
		\[ \lambda'T(s_1)Kr_1 = s_1\lambda K r_1 = s_1r_1 = 1. \]
For the linearity conditions, by Propositions \ref{examplesLinearMorphisms} and \ref{constructingLinearMorphisms}, the following composites are linear:
	\[ T({\sf q'}) \to^{(T(s_1),T(s_0))} T({\sf q}) \to^{(K,p)} {\sf q} \to^{(r_1,r_0)} {\sf q'} \]
	\[ {\sf p_{E'}} \to^{(T(s_1),s_1)} {\sf p_E} \to^{(K,q)} {\sf q} \to^{(r_1,r_0)} {\sf q'} \]
But by naturality of $p$, the first composite is $(T(s_1)Kr_1,p)$, and by the bundle condition on $s$, the second composite is $(T(s_1)Kr_1,q')$, so we have the required linearity conditions.  
\end{proof}

A useful corollary of this last result is the following.

\begin{corollary}\label{retractAffineVD}
If $K$ is an affine vertical connection on $M$, and $M'$ is a retract of $M$, $\rhd^s_r : M' \to M$, then $T^2(s)KT(r)$ is an affine vertical connection on $M'$.
\end{corollary}
\begin{proof}
By Proposition \ref{examplesLinearMorphisms}, applying $T$ to the section/retraction pair $(s,r)$ makes $(T(s),s)$ a section for $(T(r),r)$ in the category of differential bundles and linear maps; we then apply the previous result.
\end{proof}

The results above give one way of proving the classical result that every smooth manifold has an affine connection (see, for example, \cite[pg. 132]{diffGeoDynamicalSystems}).  Every smooth manifold is a smooth retract of an open subset of some $\mathbb{R}^n$ \cite[pg. 267]{lawvereToposesOfGraphs}; each of these open subsets have a canonical choice of vertical connection as described in example \ref{exampleTrivialDiffBundleVD}.  Thus by Corollary \ref{retractAffineVD}, every smooth manifold inherits a canonical choice of affine vertical connection.  

When a bundle has a vertical connection this immediately simplifies calculations for the maps involved in the universal property of the lift:

\begin{lemma}\label{lemmaKUniversality}
Suppose a differential bundle ${\sf q}$ has a vertical connection $K$.  For any $f: X \to T(E)$ with  $fT(q) = fT(q)p0$:
\begin{enumerate}[(i)]
	\item $f_{|\mu} = f\<K,p\>$;
	\item $\{f\} = fK$.
\end{enumerate}
and if $g: X \to T^2(E)$ satisfies $gT(p) = gT(p)p0$, then
\begin{enumerate}[(i)]
	\setcounter{enumi}{2}
	\item $\{g\}K = gT(K)K$.
\end{enumerate}
\end{lemma}
\begin{proof}
\begin{enumerate}[(i)]
\item $ \mu \<K,p\> = 1 \Rightarrow f_{|\mu} = f_{|\mu}\mu \<K,p\> =  f \<K,p\>$. 
\item $\{f\} = f_{|\mu}\pi_0 = f\<K,p\>\pi_0 = fK$. 
\item Since $K$ is a linear morphism, 
	$ \{g\}K = \{gT(K)\}$
by \cite[Lemma 2.12.ii]{diffBundles}.  But then by part (ii) of this lemma, $\{gT(K)\} = gT(K)K$.  
\end{enumerate}
\end{proof}

One standard way of defining a connection on a vector bundle $\sf q$ over $M$ involves giving a covariant derivative: an operation which takes as input a vector field on $M$ (that is, a section of $p: T(M) \to M$) and a section of $q$, and produces a section of $q$, satisfying various axioms.  Any vertical connection $K$ gives rise to such an operation.

\begin{definition}
For a differential bundle $\sf q$, let $\chi({\sf q})$ denote the sections of $q$.  
\end{definition}

\begin{definition}
Let $K$ be a vertical connection on a differential bundle $\sf q$.  The \textbf{covariant derivative} for $K$ is an operation
	\[ \nabla_K: \chi({\sf p_M}) \times \chi({\sf q}) \to \chi({\sf q}) \]
given by mapping $(w: M \to T(M), s: M \to E)$ to
	\[ \nabla_K (w,s) := M \to^{w} T(M) \to^{T(s)} T(E) \to^{K} E \]
\end{definition}
Note that $\nabla(w,s) \in \chi({\sf q})$ since
	\[ \nabla_K(w,s) = T(s)Kq = wT(s)T(q)p = wT(sq)p = wp = 1. \]
As mentioned earlier, in the category of smooth manifolds to give such an operation (satisfying the appropriate axioms) is equivalent to giving a vertical connection (see \cite{patterson}).  However, such a result does not hold in a general tangent category.

%%%%%%%%%%%%%%%%%%%%%%%%%%%%%%%%%%%%%%%%%%%%%%%%%%%

\subsection{Curvature}

Classically, every connection has an associated notion of curvature, described as an operation involving vector fields and sections of the differential bundle.   As discussed in \cite{patterson}, for a vertical connection on a differential bundle on $E$, a more appropriate notion of curvature is found by considering certain maps with domain $T^2(E)$.  We follow this idea to describe curvature of a connection in a tangent category.  

\begin{definition}
In a tangent category with a vertical connection $K$ on a differential bundle $\sf q$, define the \textbf{abstract curvature} of $K$ to be the map
	\[ D_K := \<cT(K)K,T(K)K\>: T^2(E) \to E_2 \]
Say that the vertical connection is \textbf{flat} if the abstract curvature $D_K$ factors through the diagonal (that is, if $cT(K)K = T(K)K$).  
\end{definition}

\begin{example}{\em ~
The canonical vertical connection on a differential object (Proposition \ref{diffObjectVD}) is flat, since by Proposition 3.6 of \cite{diffBundles}, $cT(\hat{p})\hat{p} = T(\hat{p})\hat{p}$.
}
\end{example}

Flatness is preserved by applying $T$ and by pulling back:

\begin{proposition} Suppose that $K$ is a flat vertical connection on $\sf q$.  Then:
\begin{enumerate}[(i)]
	\item $cT(K)$ is a flat vertical connection on $T({\sf q})$;
	\item for any $f: X \to M$, $f^*(K)$ is a flat vertical connection on $f^*({\sf q})$.
\end{enumerate}
\end{proposition}
\begin{proof}
\begin{enumerate}[(i)]
	\item Proposition \ref{tVertDescent} established that $cT(K)$ is a vertical connection.  To show that it is flat if $K$ is, we calculate
	\begin{eqnarray*}
	&   & cT(cT(K))cT(K) \\
	& = & cT(c)T^2(K)cT(K) \\
	& = & cT(c)cT^2(K)T(K) \mbox{ (naturality of $c$)} \\
	& = & T(c)cT(c)T^2(K)T(K) \mbox{ (coherence of $c$)} \\
	& = & T(c)cT(cT(K)K) \\
	& = & T(c)cT(T(K)K) \mbox{ (flatness of $K$)} \\
	& = & T(c)cT^2(K)T(K)  \\
	& = & T(c)T^2(K)cT(K) \mbox{ (naturality of $c$)} \\
	& = & T(cT(K))cT(K)
	\end{eqnarray*}
	as required.  
	
	\item Proposition \ref{pullbackVD} established that $f^*(K) = K \times p$ is a vertical connection.  To show that it is flat if $K$ is, we calculate
	\begin{eqnarray*}
	&   & cT(K \times p)(K \times p) \\
	& = & c(T(K) \times T(p)(K \times p) \\
	& = & cT(K)K \times cT(p)p \\
	& = & T(K)K \times pp \mbox{ (flatness of $K$ and $c$ a bundle morphism)} \\
	& = & T(K)K \times T(p)p \mbox{ (naturality of $p$)} \\
	& = & T(K \times p)(K \times p)
	\end{eqnarray*}
	as required.
\end{enumerate}
\end{proof}

In particular, by (ii), the canonical affine vertical connection on a differential object (Example \ref{exampleTrivialDiffBundleVD}) is flat.  

In general, the retract of a flat vertical connection (see Proposition \ref{retractVD}) is not necessarily flat, since
	\[ cT(T(s_1)Kr_1)T(s_1)Kr_1 = cT^2(s_1)T(K)T(r_1)T(s_1)Kr_1 = T^2(s_1)cT(K)T(r_1s_1)Kr_1 \]
so that the idempotent $r_1s_1$ prevents bringing the $cT(K)$ and $K$ terms together.  

If the tangent category has negatives, we can take the difference of the two terms in the abstract curvature, giving the notion of curvature considered in \cite{patterson}.

\begin{definition}
In a tangent category with negatives with a vertical connection $K$ on a differential bundle $\sf q$, define the \textbf{curvature} of $K$ to be the map
	\[ C_K := cT(K)K - T(K)K: T^2(E) \to E \]
\end{definition}
Note that in this case, $K$ is flat if and only $C_K$ is identically zero (that is, $C_K = pp\zq$). 

Curvature is typically seen as a tensor; that is, an operation on vector fields of $M$ and sections of $E$.  The general notion considered above can be used to derive such an operation:
\begin{definition}
The \textbf{curvature tensor} $R_K$ of a vertical connection $K$ is an operation which takes two vector fields $w_1,w_2$ on $M$ and a section $s: M \to E$ of $q$ and produces a section of $E$ by
	\[ R_K(w_1,w_2,s) := w_2T(w_1)T^2(s)C_K. \]
\end{definition}
Note that if $K$ is flat then $R_K$ is identically zero.  

The following result shows that this derived operation is given by the standard expression in terms of the covariant derivative:

\begin{proposition}\label{standardCurvProp}
$R$ is given by the standard expression for the curvature tensor (eg., \cite[pg. 239]{spivak2}); that is, 
	\[ R_K(w_1,w_2,s) = \nabla_K(w_1,\nabla(w_2,s)) - \nabla_K(w_2,\nabla(w_1,s)) - \nabla_K([w_1,w_2],s) \]
(where $[w_1,w_2]$ is defined in \cite[Lemma 3.13]{sman3}).  
\end{proposition}
\begin{proof}
First consider
\begin{eqnarray*}
	&   & \nabla_K([w_1,w_2],s) \\
	& = & \{w_1T(w_2) - w_2T(w_1)c\}T(s)K \\
	& = & \{(w_1T(w_2) - w_2T(w_1)c)T^2(s)\}K \mbox{ (by \cite[Lemma 2.12.ii]{diffBundles}} \\
	& = & (w_1T(w_2) - w_2T(w_1)c)T^2(s)T(K)K \mbox{ (by Lemma \ref{lemmaKUniversality})} \\
	& = & w_1T(w_2)T^2(s)T(K)K - w_2T(w_1)cT^2(s)T(K)K \mbox{ ($T^2(s), T(K),K$ all additive)} \\
	& = & w_1T(w_2)T^2(s)T(K)K - w_2T(w_1)T^2(s)cT(K)K \mbox{ (naturality of $c$)} \\
\end{eqnarray*}
Thus
\begin{eqnarray*}
&   & \nabla_K(w_1,\nabla(w_2,s)) - \nabla_K(w_2,\nabla(w_1,s)) - \nabla_K([w_1,w_2],s) \\
& = & w_1T(w_2)T^2(s)T(K)K - w_2T(w_1)T^2(s)T(K)K\\
& & ~~~~~ - (w_1T(w_2)T^2(s)T(K)K - w_2T(w_1)T^2(s)cT(K)K) \\
& = & w_2T(w_1)T^2(s)cT(K)K - w_2T(w_1)T^2(s)T(K)K \\
& = & w_2T(w_1)T^2(s)(cT(K)K - T(K)K) \\
& = & w_2T(w_1)T^2(s)C_K \\
& = & R_K(w_1,w_2,s)
\end{eqnarray*}
\end{proof}

%%%%%%%%%%%%%%%%%%%%%%%%%%%%%%%%%%%%%%%%%%%%%%%%%%%%%%%%%%%%%%%%%%

\subsection{Torsion and the Bianchi identities}

When the vertical connection is affine, that is, when $\sf q$ is the tangent bundle on some object $M$ (and hence $K: T^2(M) \to T(M)$), then one can also consider appropriate generalizations of the notion of torsion.  As with curvature in a tangent category, the torsion of an affine vertical connection is better seen as a map involving $T^2(M)$, rather than as an operation on vector fields.  However, as we shall see, just as for curvature, the vector field operation can be derived from the general notion.  

\begin{definition}
In a tangent category  define the \textbf{abstract torsion} of $K$, for an affine vertical connection $K: T^2(M) \to T(M)$, to be the map
	\[ W_K := \< cK,K\>: T^2(M) \to^{} T_2(M) \]
Say that the vertical connection is \textbf{torsion-free} if the torsion factors through the diagonal (that is, if $cK = K$).  
\end{definition}

\begin{example}{\em ~
The canonical affine vertical connection on a differential object $A$ (Example \ref{exampleTrivialDiffBundleVD}) is torsion-free, since 
	\[ c\<T(\hat{p})\hat{p},pp\> = \<cT(\hat{p})\hat{p},cpp\> = \<T(\hat{p})\hat{p},T(p)p\> = \<T(\hat{p})\hat{p},pp\> \]
using Proposition 3.6 of \cite{diffBundles} and naturality of $p$.
}
\end{example}

\begin{example}{\em ~
An affine connection on a Euclidean space (Example \ref{exampleVD}.1) is torsion-free if and only if the Christoffel symbols of the second kind are symmetric; that is, if and only if for all $i,j,k$, $\Gamma^i_{jk} = \Gamma^{i}_{kj}$. 
}
\end{example}

\begin{example}{\em ~
The canonical vertical connection on the sphere (Example \ref{exampleVD}.2) is torsion-free since
	\[ w + (v \cdot y^T)x = w + (y\cdot v^T)x. \]
}
\end{example}

Torsion-freeness is preserved by retraction:

\begin{proposition}
Suppose that $K$ is a torsion-free affine vertical connection on $M$, and $M'$ is a retract of $M$ via the maps $s: M' \to M$, $r: M \to M'$.   Then $T^2(s)KT(r)$ is a torsion-free affine vertical connection on $M'$. 
\end{proposition}
\begin{proof}
Proposition \ref{retractAffineVD} proves that $T^2(s)KT(r)$ is an affine vertical connection on $M'$.  To show that it is torsion-free if $K$ is straightforward using naturality of $c$:
	\[ cT^2(s)KT(r) = T^2(s)cKT(r) = T^2(s)KT(r). \]
\end{proof}

\begin{definition}
In a tangent category with negatives, given an affine vertical connection $K: T^2(M) \to T(M)$, define the \textbf{torsion} of $K$ to be the map
	\[ V_K := cK - K: T^2(M) \to T(M)\]
\end{definition}
Note that in this case $K$ is torsion-free if and only if $V_K$ is identically zero (that is, $V_K = pp0$).

\begin{definition}
For an affine vertical connection $K$, the \textbf{torsion tensor} of two vector fields $w_1,w_2: M \to T(M)$ is defined by
	\[ T_K(w_1,w_2) := w_2T(w_1)V_K. \]
\end{definition}

This agrees with the standard expression of the torsion tensor (eg., see \cite[pg. 236]{spivak2}):
\begin{proposition}
If $K$ is an affine connection with associated torsion $V_K$ then
	\[ T_K(w_1,w_2) = \nabla_K(w_1,w_2) - \nabla_K(w_2,w_1) - [w_1,w_2]. \]
\end{proposition}
\begin{proof}
First consider
\begin{eqnarray*}
&   & [w_1,w_2] \\
& = & \{w_1T(w_2) - w_2T(w_1)c\} \\
& = & (w_1T(w_2) - w_2T(w_1)c)K \mbox{ (by lemma \ref{lemmaKUniversality})} \\
& = & w_1T(w_2)K - w_2T(w_1)cK \mbox{ ($K$ is additive)} \\
\end{eqnarray*}
So that
\begin{eqnarray*}
&   & \nabla_K(w_1,w_2) - \nabla_K(w_2,w_1) - [w_1,w_2] \\
& = & w_1T(w_2)K - w_2T(w_1)K - (w_1T(w_2)K - w_2T(w_1)cK) \\
& = & w_2T(w_1)cK - w_2T(w_1)K \\
& = & w_2T(w_1)(cK - K) \\
& = & T_K(w_1,w_2)
\end{eqnarray*}
\end{proof}

As a corollary we obtain another standard result:

\begin{corollary}\label{corConnectionLieBracket}
If $K$ is torsion-free, then
	\[ [w_1,w_2] = \nabla_K(w_1,w_2) - \nabla_K(w_2,w_1). \]
\end{corollary}
\begin{proof}
By the previous result, if $K$ is torsion-free, then
\begin{eqnarray*}
&   & \nabla_K(w_1,w_2) - \nabla_K(w_2,w_1) - [w_1,w_2] \\
& = & w_2T(w_1)V_K \\
& = & w_2T(w_1)pp0 \\
& = & 0
\end{eqnarray*}
so that $[w_1,w_2] = \nabla(w_1,w_2) - \nabla(w_2,w_1)$.
\end{proof}

If we have a torsion-free affine vertical connection, we can give an alternative expression for the curvature tensor: 

\begin{proposition}
If $K$ is a torsion-free affine connection, then 
	\[ R_K(w_1,w_2,s) = \nabla_K^2(w_1,w_2,s) - \nabla_K^2(w_2,w_1,s) \]
where
	\[ \nabla_K^2(w_1,w_2,s) := \nabla_K(w_1,\nabla_K(w_2,s)) - \nabla_K(\nabla_K(w_1,w_2),s). \]
\end{proposition}
\begin{proof}
By using proposition \ref{standardCurvProp}, it suffices to show that
	\[ \nabla_K([w_1,w_2],s) = \nabla_K(\nabla_K(w_1,w_2),s) - \nabla_K(\nabla_K(w_2,w_1),s). \]
Indeed,
\begin{eqnarray*}
&   & \nabla_K(\nabla_K(w_1,w_2),s) - \nabla_K(\nabla_K(w_2,w_1),s) \\
& = & w_1T(w_2)KT(s)K - w_2T(w_1)KT(s)K \\
& = & w_1T(w_2)KT(s)K - w_2T(w_1)cKT(s)K \mbox{ (since $K$ torsion-free)} \\
& = & (w_1T(w_2) - w_2T(w_1)c)KT(s)K \\
& = & [w_1,w_2]T(s)K \mbox{ (by corollary \ref{corConnectionLieBracket}}) \\
& = & \nabla_K([w_1,w_2],s)
\end{eqnarray*}
\end{proof}

Returning to the more general curvature $C_K = cT(K)K - T(K)K$, we can prove several important identities for this definition when $K$ is a torsion-free affine vertical connection.   
\begin{theorem}
In a tangent category with negative, if $K: T^2(M) \to T(M)$ is a torsion-free affine vertical connection and $C = C_K$ is its associated curvature then
\begin{enumerate}[(i)]
	\item (Anti-symmetry) $cC = -C$;
	\item (First Bianchi identity) 	
		\[ C + cT(C)C + T(c)cC = 0. \]
	\item (Second Bianchi identity) 	
		\[ D_K(C) + T(c)T^2(c)D_K(C) + T^2(c)T(c)D_K(C) = 0, \] 
	where $D_K(C) = T(C)K$.  
\end{enumerate}
\end{theorem}
\begin{proof}
\begin{enumerate}[(i)]
	\item Follows directly from $c^2 = 1$.
	\item Consider
	\begin{eqnarray*}
	&   & C + cT(c)C + T(c)cC \\
	& = & (cT(K)K - T(K)K) + cT(c)(cT(K)K - T(K)K) + T(c)c(cT(K)K - T(K)K) \\
	& = & cT(K)K - T(K)K + cT(c)cT(K)K - cT(c)T(K)K + T(c)T(K)K - T(c)cT(K)K \\
	& = & cT(K)K - T(K)K + T(c)cT(c)T(K)K \\ & & ~~~~~ - cT(c)T(K)K + T(c)T(K)K - T(c)cT(K)K \\
	& = & cT(K)K - T(K)K + T(c)cT(K)K \\ & & ~~~~~ - cT(K)K + T(K)K - T(c)cT(K)K \mbox{ (since $K$ torsion-free)} \\
	& = & 0
	\end{eqnarray*}
	\item For simplicity, let $Q = T^2(K)T(K)K$, so $D_K(C) = T(c)Q - Q$.  Note that since $K$ is torsion-free, $T^2(c)Q = Q$.  Then consider
	\begin{eqnarray*}
	&   & D_K(C) + T(c)T^2(c)D_K(C) + T^2(c)T(c)D_K(C) \\
	& = & (T(c)Q - Q) + T(c)T^2(c)(T(c)Q - Q) + T^2(c)T(c)(T(c)Q - Q) \\
	& = & T(c)Q - Q + T(c)T^2(c)T(c)Q - T(c)T^2(c)Q + T^2(c)Q - T^2(c)T(c)Q \\
	& = & T(c)Q - Q + T^2(c)T(c)T^2(c)Q - T(c)T^2(c)Q + T^2(c)Q - T^2(c)T(c)Q \\
	& = & T(c)Q - Q + T^2(c)T(c)Q - T(c)Q + Q - T^2(c)T(c)Q \mbox{ (since $T^2(c)Q = Q$)} \\
	& = & 0
	\end{eqnarray*}
\end{enumerate}
\end{proof}

%%%%%%%%%%%%%%%%%%%%%%%%%%%%%%%%%%%%%%%%%%%%%%%%%%%%%

\subsection{Finsler connections}

In this section, we shall briefly investigate an alternative way of defining a vertical connection. It involves the bundle $q^{*}({\sf q})$, which is referred to as the Finsler bundle in \cite{SLK}.  The  Finsler bundle is the differential bundle structure induced on $\pi_1: E_2 \to E$ as the bundle obtained by pulling back ${\sf q}$ along $q$.  The lift $\lambda_{q^{*}({\sf q})}$ of this bundle is given as the unique map  fitting in:
$$\xymatrix{ & T(E_2) \ar[dd]^<<<<<<<{T(\pi_1)}|\hole  \ar[rr]^{T(\pi_0)}  && T(E) \ar[dd]^{T(q)} \\
                   E_2 \ar@{..>}[ur]^{\lambda_{q^{*}({\sf q})}} \ar[dd]_{\pi_1} \ar[rr]_>>>>>>>{\pi_0}  & & E \ar[dd]^<<<<<<<{q} \ar[ru]^{\lambda} \\
                     & T(E) \ar[rr]_<<<<<<<{T(q)}|\hole  & & T(M) \\
                    E \ar[ru]_{0} \ar[rr]_{q} & & M \ar[ru]_{0}}$$
and so $\lambda_{q^{*}({\sf q})} := \< \pi_0\lambda,\pi_10\> = \lambda \x_M 0$.  Clearly the additive structure of this bundle is given, as expected, by addition on the first coordinate.
 
 We observe that there is the following important linear bundle morphism from the Finsler bundle to the tangent bundle of $E$:
 
 \begin{lemma}\label{muLinearity} 
For any differential bundle ${\sf q}$ 
           $$(\mu, 1_E) : q^{*}({\sf q}) \to {\sf p}_E$$
is a linear bundle morphism.
 \end{lemma}
 
 \begin{proof}
 We must show this is a bundle morphism that is
 $$\xymatrix{E_2 \ar[d]_{\pi_1} \ar[rr]^{\mu} & & T(E) \ar[d]^{p} \\ E \ar@{=}[rr] & & E}$$ 
 commutes.  However, $\mu p = \pi_1$ is a basic identity of differential bundles (see \cite{diffBundles} Lemma 2.7).
 
 Thus, it remain to show that this is a linear bundle morphism (that is, $\lambda_{\sf f} T(\mu) = \mu \ell$) which we calculate out as follows:
 \begin{eqnarray*}
 \lambda_{\sf f} T(\mu) & = & \<\pi_0\lambda,\pi_10\>T(\<\pi_0\lambda,\pi_10\>T(\pq)) \\
 & = & \<\pi_0\lambda T(\lambda),\pi_10T(0)\>T^2(\pq) \\
 & = & \< \pi_0 \lambda\ell,\pi_10\ell\> T^2(\pq) \\
 & = & \< \pi_0 \lambda,\pi_10\> T(\pq)\ell \\
 & = & \mu \ell.
 \end{eqnarray*}
 \end{proof}

Now the object $E_2$ is a bundle in a rather different way because it is the apex of the pullback given by the universality of the lift:
$$\xymatrix{E_2 \ar[d]_{\pi_0q=\pi_1q} \ar[rr]^{\mu} & & T(E) \ar[d]^{T(q)} \\ M \ar[rr]_{0} && T(M) }$$
That is, it is also the object of the pullback bundle $0^*(T({\sf q}))$.  It is useful to recall the structure of this bundle.  Note that it is not immediate what the bundle addition is.  By definition it is inherited from the bundle $T({\sf q})$ pulling back the addition $T(\pq)$:  but what actually is this?  Fortunately, there is a simple description: one way to determine this structure is to look at the lift map of this bundle.  This is the unique map $\lambda_2$ in the diagram below:
$$\xymatrix{ & T(E_2) \ar[rr]^{T(\mu)} \ar[dd]^<<<<<<<{T(\pi_0q)}|\hole & & T^2(E) \ar[dd]^{T(q)} \\
                    E_2 \ar@{..>}[ur]^{\lambda_2} \ar[rr]^>>>>>>>{\mu} \ar[dd]_{\pi_0q} & & T(E) \ar[dd]^<<<<<<<{T(q)} \ar[ur]_{T(\lambda)c} \\
                    &  T(M) \ar[rr]_<<<<<<<{T(0)}|\hole & & T^2(M) \\
                    M \ar[rr]_{0} \ar[ur]^{0} & & T(M) \ar[ur]_{0}}$$
 Thus, $\lambda_2$ is the  unique map such that  $\lambda_2 T(\pi_0q) = \pi_0q0$ and $\lambda_2T(\mu) = \mu T(\lambda)c$.
 
 \begin{lemma} 
 The bundle $0^{*}(T({\sf q}))$ is just the product bundle of ${\sf q}$ with itself  in the fiber over $M$.  In particular, $\lambda_2 = \< \pi_0 \lambda,\pi_1 \lambda\>$.
 \end{lemma}
 
 \begin{proof} 
 It suffices to show that $\lambda_2 = \< \pi_0 \lambda,\pi_1 \lambda\>$ works in the equation above as then the fact that it is the product 
 bundle is immediate.  We have:
 \begin{eqnarray*}
 \< \pi_0 \lambda,\pi_1 \lambda\> T(\pi_0q) & = & \pi_0 \lambda T(q) = \pi_0 q 0 = \pi_1 q 0 \\
 \< \pi_0 \lambda,\pi_1 \lambda\>T(\mu) & = & \< \pi_0 \lambda,\pi_1 \lambda\>T(\< \pi_0 \lambda,\pi_10 \> T(\pq)) \\
       & = & \< \pi_0 \lambda T(\lambda), \pi_1\lambda T(0) \> T^2(\pq) \\
       & = & \< \pi_0 \lambda\ell, \pi_1\lambda T(0) \> T^2(\pq) \\
        & =  &  \< \pi_0 \lambda \ell c ,\pi_1 \lambda 0 c\> T^2(\pq) \\
        & =  &  \< \pi_0 \lambda T(\lambda)c ,\pi_1 0 T(\lambda) c\> T^2(\pq) \\
        & =  &  \< \pi_0 \lambda T(\lambda) ,\pi_1 0 T(\lambda)\> T^2(\pq)c \\
        & =  &  \< \pi_0 \lambda,\pi_1 0\> T(\lambda_2 T(\pq))c \\
        & =  &  \< \pi_0 \lambda,\pi_1 0\> T(\pq \lambda)c \\
        & =  &  \mu T(\lambda)c.
 \end{eqnarray*}
 \end{proof}

We are now ready to give an alternative presentation of a vertical connection which we shall refer to as a \textbf{Finsler connection} on a differential bundle.  

\begin{definition}
A {\bf Finsler connection} is given by a map $R: T(E) \to E_2$ such that 
\begin{enumerate}[{\bf [c.1]}]
\item $R$ is a retraction of $\mu$.  
\item $(R,p): T({\sf q}) \to 0^*(T({\sf q}))$ is a linear bundle morphism. 
%which is a retraction of the linear bundle morphism $(\mu,0): 0^*(T({\sf q})) \to T({\sf q})$;
\item $(R,1): {\sf p}_E \to q^*({\sf q})$  is a linear bundle morphism.
% to the Finsler bundle which is a retraction of $(\mu,1): q^*({\sf q}) \to {\sf p}_E$.
\end{enumerate}
\end{definition}

Again we may unwrap this definition into equational requirements.  Requiring that these are bundle morphisms means that $R \pi_0q = T(q)p$ and 
that $R \pi_1 = p$.  Notice that the second implies the first as 
$$R \pi_0q = R \pi_1q = pq = T(q)p.$$  
Requiring that these are retracts amounts to requiring only that $\mu R = 1$.   Requiring that the maps be linear require respectively that 
$T(\lambda)cT(R) = R \lambda_2$ and that $\ell T(R) = R \lambda_{\sf f}$.  We may express these by looking at components.  Examining the first 
equation yields:
\begin{eqnarray*}
T(\lambda)cT(R) & = & T(\lambda)c \< T(R\pi_0),T(R\pi_1) \> = \< T(\lambda) cT(R\pi_0),T(\lambda)cT(p) \> \\
& = & \< T(\lambda) cT(R\pi_0),T(\lambda)p \> =  \< T(\lambda) cT(R\pi_0),p \lambda \> \\
R \lambda_2 & = & \< R\pi_0,R\pi_1\>\< \pi_0\lambda,\pi_1\lambda\> = \< R\pi_0\lambda,p\lambda\> \\
\end{eqnarray*}
So that this amounts to the condition $R\pi_0\lambda =T(\lambda) cT(R\pi_0)$.   Examining the second gives:
\begin{eqnarray*}
\ell T(R) & = & \ell\< T(R\pi_0),T(R\pi_1)\> = \< \ell T(R\pi_0),\ell T(p)\> = \< \ell T(R\pi_0),p0 \> \\
R q^*(\lambda) & = & \< R\pi_0,R \pi_1\>\<\pi_0\lambda,\pi_00\> = \< R\pi_0\lambda, p0 \>
\end{eqnarray*}
So that this amounts to the requirement that $\ell T(R\pi_0) = R\pi_0\lambda$.  

Finally that these are retractions amounts to the requirement that $\mu R = 1_{E_2}$.  

We collect these observations into a list of equations a Finsler connection must satisfy:

\begin{lemma}\label{lemmaFinslerProps}
$R: T(E) \to E_2$ is a Finsler connection for the bundle ${\sf q}$ if and only if the following equations are satisfied:
\begin{enumerate}[(a)]
\item  $\mu R = 1_{E_2}$;
\item  $R \pi_1 = p$;
\item  $R\pi_0\lambda =T(\lambda) cT(R\pi_0)$;
\item  $\ell T(R\pi_0) = R\pi_0\lambda$.
\end{enumerate}
\end{lemma}

The objective of breaking down these conditions is to help establish:

\begin{theorem}\label{thmFinslerEquivalence}
To have a Finsler connection is precisely to have a vertical connection.  That is, if $R$ is a Finsler connection then $R \pi_0$  is a vertical connection and, conversely, if $K$ is a vertical connection then $\< K,p\>$ is a Finsler connection and these translations are inverse.
\end{theorem}

\begin{proof}  As $R\pi_1 =p$ it is clear the translations are inverse.

Starting with a Finsler connection $R$, we define $K := R\pi_0$ and must show it is a vertical connection.  First we require that 
$\lambda R \pi_0 = 1_E$  which follows as:
\begin{eqnarray*}
\lambda R \pi_0 & = & \<\lambda,\lambda T(q \zq) \> T(\pq) R  \pi_0 \\
                          & = & \< \lambda, \lambda T(\lambda p) \> T(\pq) R \pi_0 \\
                          & = & \< \lambda, \lambda \ell T(p) \> T(\pq) R \pi_0 \\
                          & = & \< \lambda, \lambda p 0 \> T(\pq) R \pi_0 \\
                          & = & \< \lambda, q \zq 0 \> T(\pq) R \pi_0 \\
                          & = & \< 1,q \zq\> \< \pi_0 \lambda, \pi_1 0 \> T(\pq) R \pi_0 \\
                          & = & \< 1,q \zq\>\mu R \pi_0 \\
                         & = & \< 1,q \zq\> \pi_0 = 1
\end{eqnarray*}
Because $R \pi_0 q = R \pi_1 q = p q = T(q) p$ both the bundle morphisms of {\bf [C.1]} and {\bf [C.2]} are present.  We must show that they are linear. 
In fact, {\bf [C.1]} is given immediately by the composite of linear bundle morphisms
$$T({\sf q}) \to^{(R,p)} 0^{*}(T({\sf q})) \to^{(\pi_0,1_M)} {\sf q}$$
and {\bf [C.2]} by the composite 
$${\sf p} \to^{(R,1)} q^{*}({\sf q}) \to^{(\pi_0,q)} {\sf q}.$$

For the converse we set $R := \< K,p\>: T(E) \to E_2$: this is well-defined as $K q = p q$ (from {\bf [C.2]}).  We need to show that  $\< K,p\>$ is a retract of $\mu$.  Indeed, using lemma \ref{lemmaVConnector} and the definition of $\mu$, we have
	\[ \mu\<K,p\> = \<\mu K, \mu p\> = \<\pi_0,\pi_1\> = 1. \]
The fact that the bundle morphisms are present for {\bf [c.2]} and {\bf [c.3]} is immediate: the difficulty is to show that they are linear.  However,  inspecting the requirements given in the lemma above shows that when $R = \< K,p\>$ they reduce to the requirements for a vertical connection.  This completes the proof.
\end{proof}

Proposition \ref{diffObjectVD} showed that in a Cartesian tangent category, every differential object $A$  has a vertical connection $\hat{p}: T(A) \to A$.  A  differential object is equivalently a differential bundle over the final object and this means that $\mu: A \times A \to T(A)$ is an isomorphism.  The results of this section show that a vertical connection is given by a retract of this map, so it follows that there can be at most one vertical connection for a differential object, namely $\mu^{-1}$.   Thus we have:

\begin{corollary}\label{diffObjectVDUnique}
In a Cartesian tangent category each differential object $A$ has a unique vertical connection given by its principal projection, $\hat{p}: T(A) \to A$.
\end{corollary}

%%%%%%%%%%%%%%%%%%%%%%%%%%%%%%%%%%%%%%%%%%%%%%%%%%%%%%%%%%%%%%%%%%%%%%%%%%%%%%%%%

\section{Horizontal connections}\label{secHorizConn}

%%%%%%%%%%%%%%%%%%%%%%%%%%%%%%%%%%%%%%%%%%%%%%%%%%%%%%%%%%%%%%%%%%%%%%%%%%%%%%%%%

One of the other standard ways of describing a connection on a bundle $\sf q$, originally due to Erhesmann, is in terms of a certain ``lifting'' of a tangent vector on $M$ and an element of $E$ to a tangent vector of $E$.  For example, this was the definition first proposed for connections in the context of Synthetic Differential Geometry (see definition 1.1 of \cite{kockConnections}).    To understand this idea in the context of tangent categories, we must first describe the bundles of such objects and their properties.  

%%%%%%%%%%%%%%%%%%%%%%%%%%%%%%%%%%%%%%%%%%%%%%%%%%%%%%%%%%%%%%%%%%%%%%%%%%%%%%%%%

\subsection{The horizontal descent}

For any differential bundle $\sf q$, we have the following pullback:
\[
\bfig
	\square<600,350>[T(M) \times_M E`T(M)`E`M;\pi_0`\pi_1`p_M`q]
\efig
\]
This comes with a canonical map
	\[ U := \<Tq,p\>: T(E) \to T(M) \times_M E \]
which we refer to as the \textbf{horizontal descent}.  

Now $T(E)$ can be viewed as a differential bundle in two different ways: as a bundle over $E$ (ie., the differential bundle $\sf p_E$) and as a bundle over $T(M)$ (ie., the differential bundle $T({\sf q})$).  Similarly, the pullback $T(M) \times_M E$ can be viewed as a differential bundle in two different ways: one as the pullback of $\sf p_M$ along $q$, ie., $q^*({\sf p_M})$, and the other as the pullback of $\sf q$ along $p$, ie., $p^*({\sf q})$.  $U$ is linear for each of these structures respectively.  

\begin{lemma}\label{UisLinear}
For a differential bundle $\sf q$, with the horizontal descent $U$ as defined above,
\begin{itemize}
	\item $(U,1_E)$ is a linear bundle morphism from $\sf p_E$ to $q^*({\sf p_M})$.
	\item $(U,1_{T(M)})$ is a linear bundle morphism from $T({\sf q})$ to $p^*({\sf q})$.
\end{itemize}
\end{lemma}
\begin{proof}
The first is a bundle morphism by definition of $U$, since $U\pi_1 = \<T(q),p\>\pi_1 = p$.  For linearity,
	\[ \ell_ET(U) = \ell_E\<T^2(q),T(p)\> = \<\ell_E T^2(q),\ell_E T(p)\> = \<T(q)\ell_M,p0\> \] \
by naturality and coherence of $\ell$.

The second is also a bundle morphism by definition of $U$, and is linear since
\begin{eqnarray*}
T(\lambda)cT(U) & = & T(\lambda)c\<T^2(q),T(p)\> \\
& = & \<T(\lambda)T^2(q)c,T(\lambda)cT(p)\> \mbox{ (using naturality of $c$)} \\
& = & \<T(\lambda T(q))c,T(\lambda)p\> \\
& = & \<T(q0))c,p\lambda\> \\
& = & \<T(q)T(0)c,p\lambda\> \\
\end{eqnarray*}
as required.  
\end{proof}

The following will be useful when calculating with $U$.  

\begin{lemma}\label{lemmaBundleMorphismsU}
For a bundle morphism $(f_1,f_0): {\sf q} \to {\sf q'}$, $T(f_1)U' = U(T(f_0) \times f_1)$.
\end{lemma}
\begin{proof}
We have
	\begin{eqnarray*} T(f_1)U' & = & T(f_1)\<T(q'),p\> = \<T(f_1q'),T(f)p\> = \<T(qf_0),pf_0\> \\
	& = & \<T(q),p\>(T(f_0) \times f_1) = U'(T(f_0) \times f_1). \end{eqnarray*}
\end{proof}

%%%%%%%%%%%%%%%%%%%%%%%%%%%%%%%%%%%%%%%%%%%%%%%%%%%%%%%%%%%%%%%%%%%%%

\subsection{Definition and examples}

We can now give the ``lifting'' version of a connection: it is a section of the horizontal descent, $U$, which is linear in both ways:  

\begin{definition}\label{defnHorizConnection}
Let $\sf q$ be a differential bundle.  A \textbf{horizontal lift} is a map $H: T(M) \times_M E \to T(E)$ which is a section of $U = \<T(q),p\>$.  A \textbf{horizontal connection} on $\sf q$ is a horizontal lift $H$ such that 
\begin{itemize}
	\item $(H,1_E)$ is a linear bundle morphism from $q^*({\sf p_M})$ to $\sf p_E$;
	\item $(H,1_{T(M)})$ is a linear bundle morphism from $p^*({\sf q})$ to $T({\sf q})$.
\end{itemize}
\end{definition}

%See definition 7.2.1, page 357 of \cite{SLK}, where it is called an \emph{Ehresmann connection}; small difference is that they require it to be linear everywhere but 0 (not sure why).    

As with vertical connections, it is helpful to unwrap the definition into specific equational requirements.  

\begin{lemma}\label{horizLiftConditions}
To give a horizontal connection on a differential bundle $\sf q$ is to give a map 
	\[ H: T(M) \times_{M} E \to T(E) \] such that:
\begin{enumerate}[(i)]
	\item $HT(q) = \pi_0$.
	\item $Hp = \pi_1$.
	\item $H \ell = (\ell \times 0)T(H)$.
	\item $HT(\lambda)c = (0 \times \lambda)T(H)$.
\end{enumerate}
\end{lemma}
\begin{proof}
The first two conditions are the requirements that $H$ be a section of $U = \<T(q),p\>$.  These are also the exact conditions required for $(H,1)$ to be a bundle morphism between the bundles described in the definition.  The final two conditions are simply the linearity requirements for these bundles.
\end{proof}

\begin{proposition}\label{diffObjectHL}
Any differential object $A$ (viewed as a differential bundle over $1$) has a canonical horizontal connection $\pi_1 0: 1 \times A \to T(A)$.
\end{proposition}
\begin{proof}
That it is a section of $U$ is straightforward, since $\pi_0 ! = ! = \pi_0: 1 \times A \to 1$ while $\pi_1 0p = \pi_1$.  The linearity conditions also follow in a straightforward way from the tangent category axioms:
	\[ (\ell \times 0)T(\pi_1 0) = (\ell \times 0)T(\pi_1)T(0) = \pi_1 0 T(0) = \pi_1 0 \ell \]
and
	\[ (0 \times \lambda)T(\pi_10) = (0 \times \lambda)T(\pi_1)T(0) = \pi_1 \lambda T(0) = \pi_1 \lambda 0 c = \pi_1 0 T(\lambda) c. \]
\end{proof}

We shall see in Proposition \ref{propConstructingHorizLift} that in many tangent categories, each vertical connection gives a horizontal connection, giving further examples of horizontal connections.

%%%%%%%%%%%%%%%%%%%%%%%%%%%%%%%%%%%%%%%%%%%%%%%%%%%%%%%%%%%%%%%%%%%%%

\subsection{Basic properties of horizontal connections}

As with vertical connections, we can construct new horizontal connections by applying $T$, by pulling back, and by retracting.   

\begin{proposition}\label{tHorizontalLift}
If $H$ is a horizontal connection on $\sf q$, then $(c \times 1)T(H)c$ is a horizontal connection on $T({\sf q})$. 
\end{proposition}
\begin{proof}
The $U$ for $T({\sf q})$ is $\<T^2(q),p\>$.  We then check that $(c \times 1)T(H)cU = 1$ by checking each component:
	\[ (c \times 1)T(H)cT^2(q) = (c \times 1)T(H)T^2(q)c = (c \times 1)T(H T(q))c = (c \times 1)T(\pi_0)c = \pi_0 cc = \pi_0 \]
and
	\[ (c \times 1)T(H)cp = (c \times 1)T(H)T(p) = (c \times 1)T(H p) = (c \times 1)T(\pi_1) = \pi_1 \]
	
For the first linearity condition, note that $(c \times 1,1)$ is a linear bundle morphism from $(T(q))^*({\sf p_{T(M)}})$ to $T(q^*({\sf p_M}))$ since
	\[ (c \times 1)T(\ell \times 0)c = cT(\ell)c \times 0c = \ell T(c) \times 0 \]
by tangent category coherences.  Then $(c \times 1)T(H)c,1)$ is linear from $(T(q))^*({\sf p_{T(M)}})$ to $\sf p_{T(E)}$ since it is the composite of
	\[ (T(q))^*({\sf p_{T(M)}}) \to^{(c \times 1,1)} T(q^*({\sf p_M})) \to^{(T(H),T(1))} T({\sf p_E}) \to^{(c,1)} {\sf p_{T(E)}} \]
with the last two linear by Propositions \ref{constructingLinearMorphisms} and \ref{examplesLinearMorphisms} respectively.  

For the second linear condition, first note that $(c \times 1,c)$ is a linear bundle morphism from $p^*(T({\sf q}))$ to $T(p^*({\sf q}))$ since
	\[ (c \times 1)T(0 \times \lambda)c = cT(0)c \times T(\lambda)c = c0 \times T(\lambda)c = 0T(c) \times T(\lambda)c \]
by naturality of $0$.  Then $(c \times 1)T(H)c,1)$ is linear from $p^*(T({\sf q}))$ to $T(T({\sf q}))$ since it is the composite of
	\[ p^*(T({\sf q})) \to^{(c \times 1,c)} T(p^*({\sf q})) \to^{(T(H),T(1))} T(T({\sf q}) \to^{(c,c)} T(T({\sf q}) \]
with the last two linear by Propositions \ref{constructingLinearMorphisms} and \ref{examplesLinearMorphisms} respectively.  
\end{proof}

\begin{proposition}\label{pullbackHL}
Suppose that $H$ is a horizontal connection on ${\sf q}$, a differential bundle over $M$.  Then for any map $f: X \to M$, $f^{*}H$ is a horizontal connection on $f^{*}({\sf q})$, where $f^{*}H$ is defined by
\[
\xymatrix{
T(X) \times_X f^*E \ar[rr]^{T(f) \times f^*_E} \ar[dddr]_{\pi_0} \ar@{..>}[dr]^{f^*H} &  & T(M) \times_M E \ar[dr]^H & \\
& T(f^*E) \ar[rr]^{T(f^*_E)} \ar[dd]^{T(f^*(q))} &  & T(E) \ar[dd]^{T(q)} \\
& & & \\
& TX \ar[rr]_{T(f)} &  & TM }
\]
\end{proposition}
\begin{proof}
First, note that $f^*H$ is well-defined since by definition of $H$
	\[ (T(f) \times f^*_E)HT(q) = (T(f) \times f^*_E)\pi_0 = \pi_0T(f). \]
To show that $f^*H$ is a horizontal connection we will explicitly check the four conditions from Lemma \ref{horizLiftConditions}.  For the first, 
$f^*H T(f^*(q)) = \pi_0$ by definition of $f^*H$.

For the second,
\begin{eqnarray*}
&   & f^*Hp \\
& = & \<(T(f) \times f^*_E)H,\pi_0\>p \\
& = & \<(T(f) \times f^*_E)Hp,\pi_0p\> \mbox{ (by naturality of $p$)} \\
& = & \<(T(f) \times f^*_E)\pi_1,\pi_0p\> \mbox{ (by definition of $H$)} \\
& = & \<\pi_1 f^*_E,\pi_1 f^*q \> \\
& = & \pi_1 \<f^*_E,f^*(q)\> \\
& = & \pi_1
\end{eqnarray*}

For the third, 
\begin{eqnarray*}
&   & f^*H \ell \\
& = & \<(T(f) \times f^*_E)H,\pi_0\>\ell \\
& = & \<(T(f) \times f^*_E)H \ell,\pi_0 \ell\> \mbox{ (by naturality of $\ell$)} \\
& = & \<(T(f) \times f^*_E)(\ell \times 0)T(H),\pi_0 \ell\> \mbox{ (definition of $H$)} \\
& = & \<(\ell T^2(f) \times 0T(f^*_E))T(H),\pi_0 \ell \> \mbox{ (naturality of $\ell$ and $0$)} \\
& = & (\ell \times 0)\<(T^2(f) \times T(f^*_E)T(H),T(\pi_0)\> \\
& = & (\ell \times 0)T(f^*H)
\end{eqnarray*}

Finally, recall that the lift for $f^*({\sf q})$ is given by $f^*\lambda = \<f^*_E \lambda, f^*(q)0\>$.  As $f^*_E$ and $f^*(q)$ are the projections from $f^*(E)$, we can shorten this to $f^*\lambda = \lambda \times 0$.  We now check the fourth condition:
\begin{eqnarray*}
&   & f^*H T(\lambda \times 0)c \\
& = & \<(T(f) \times f^*_E)H,\pi_0\>(T(\lambda) \times T(0))c \\
& = & \<(T(f) \times f^*_E)HT(\lambda),\pi_0 T(0)\>c \\
& = & \<(T(f) \times f^*_E)HT(\lambda)c,\pi_0 c \> \mbox{ (by naturality of $c$)} \\
& = & \<(T(f )\times f^*_E)(0 \times \lambda)T(H),\pi_0 0\> \mbox{ (definition of $H$ and naturality of $0$)} \\
& = & \<0T^2(f) \times f^*_E \lambda) T(H),\pi_0 0\> \mbox{ (naturality of $0$)} \\
& = & (0 \times (\lambda \times 0))\<T^2(f) \times T(f^*_E)T(H),T(\pi_0)\> \\
& = & (0 \times (\lambda \times 0)) T(f^*H)
\end{eqnarray*}
as required.  
\end{proof}

\begin{example}\label{canonicalAffineHL}{\em ~
In particular, just as any trivial differential bundle $A \times M \to M$ (for $A$ a differential object) has an associated vertical connection (by pulling back the vertical connection on $A$) so too do such trivial differential bundles have a horizontal connection given by pulling back the horizontal connection on $A$.  In particular, this is true of the tangent bundle of a differential object $A$.  The particular form of this horizontal connection is
	\[ \<(! \times \hat{p})\pi_10,\pi_0\> = \<\pi_1 \hat{p}0,\pi_0\> \]
which has four components ($T^2(A)$ is the product of four copies of $A$):
	\[ \<!\zq,\pi_1 \hat{p},\pi_0 \hat{p}, \pi_0 p \>. \]
}
\end{example}

\begin{proposition}\label{retractHL}
Suppose that we have a section/retraction pair $\rhd^s_r: {\sf q'} \to {\sf q}$ in the category of differential bundles and linear maps:
\[
\bfig
	\morphism(0,0)|a|/{@{>}@/^1em/}/<500,0>[{\sf q'}`{\sf q};(s_1,s_0)]
	\morphism(500,0)|b|/{@{>}@/^1em/}/<-500,0>[{\sf q}`{\sf q'};(r_1,r_0)]
\efig
\]
Then if $H$ is a horizontal connection on $\sf q$, then $(T(s_0) \times s_1)HT(r_1)$ is a horizontal connection on $\sf q'$.
\end{proposition}
\begin{proof}
For the section requirement, we have
\begin{eqnarray*}
&   & (T(s_0) \times s_1)HT(r_1)\<T(q'),p\> \\
& = & (T(s_0) \times s_1)H\<T(r_1q'),T(r_1)p\> \\
& = & (T(s_0) \times s_1)H\<T(q)T(r_0),pr_1\> \\
& = & (T(s_0) \times s_1)H\<T(q),p\>(T(r_0) \times r_1) \\
& = & (T(s_0) \times s_1)(T(r_0) \times r_1) \\
& = & 1
\end{eqnarray*}
For the first linearity condition, note that $(T(s_0) \times s_1,s_1)$ is a linear bundle morphism from $(q')^*({\sf p_{M'}})$ to $q^*({\sf p_M})$ since
	\[ (T(s_0) \times s_1)(\ell \times 0) = T(s_0) \ell \times s_1 0 = \ell T^2(s_0) \times 0 T(s_1) = (\ell \times 0)T(T(s_0) \times s_1) \]
by naturality of $\ell$ and $0$.  Then $((T(s_0) \times s_1)HT(r_1),1)$ is linear from $(q')^*({\sf p_{M'}})$ to $\sf p_{E'}$ since it is the composite of
	\[ (q')^*({\sf p_{M'}}) \to^{(T(s_0) \times s_1,s_1)} q^*({\sf p_M}) \to^{(H,1)} {\sf p_E} \to^{(T(r_1),r_1)} {\sf p_{E'}} \]
with $H$ linear by assumption and $(T(r_1),r_1)$ linear by Proposition \ref{examplesLinearMorphisms}.

For the second linearity condition, note that $(T(s_0) \times s_1,Ts_0)$ is linear from $p^*({\sf q'})$ to $p^*({\sf q})$ since
	\[ (T(s_0) \times s_1)(0 \times \lambda) = T(s_0)0 \times s_1 \lambda = 0T^2(s_0) \times \lambda' T(s_1) \]
by naturality of $0$ and linearity of $s$.  Then $((T(s_0) \times s_1)HT(r_1),1)$ is linear from $p^*({\sf q'})$ to $T({\sf q'})$ since it is the composite of 
	\[ p^*({\sf q'}) \to^{(T(s_0) \times s_1,T(s_0))} p^*({\sf q}) \to^{(H,1)} T({\sf q}) \to^{(T(r_1),T(r_0))} T({\sf q'}) \]
with $H$ linear by assumption and $(T(r_1),T(r_0))$ linear by Proposition \ref{constructingLinearMorphisms}.
\end{proof}

%%%%%%%%%%%%%%%%%%%%%%%%%%%%%%%%%%%%%%%%%%%%%%%%%%%%%%%%%%%%%%%%%%%%%

\section{Connections}\label{secConn}

%%%%%%%%%%%%%%%%%%%%%%%%%%%%%%%%%%%%%%%%%%%%%%%%%%%%%%%%%%%%%%%%%%%%%

We now define a full connection as consisting of a vertical connection and a horizontal connection which are compatible.  We shall show that, in some 
circumstances, having either a vertical connection or horizontal connection is sufficient to define a full connection.  

\subsection{Definition and examples}

\begin{definition}\label{defnConnection}
A \textbf{connection}, $(K,H)$, on a differential bundle $\sf q$ consists of a vertical connection $K$ and a horizontal connection $H$ on $\sf q$ such that:
\begin{itemize}
	\item $H K = \pi_1q\zq$;
	\item $\<K,p\>\mu + UH = 1_{T(E)}$.
\end{itemize}
A connection is \textbf{affine} if the differential bundle $\sf q$ is the tangent bundle of some object $M$.  The \textbf{curvature} (for a general connection) and \textbf{torsion} (for an affine-free connection) are the curvature and torsion of $K$, and we say that the connection is \textbf{flat} or \textbf{torsion-free} if $K$ is.  
\end{definition}

%(Note: cannot have $K\lambda + UH = 1$, since can't add them! 
Note that 
	\[ \<K,p\>\mu = \<K,p\>\<\pi_0\lambda,\pi_10\>T(\pq) = \<K\lambda,p0\>T(\pq) \]
so that the second condition can be re-written in infix notation as 
	\[ \<K\lambda,p0\>T(\pq) + UH = 1_{T(E)}. \]
We will sometimes write $T(\pq)$ as $\pqone$, so that again using infix notation the equation becomes
	\[ (K\lambda \pqone p0) + UH = 1. \]

For reference, in Table \ref{dataforConn} we provide data and equations for vertical connections, horizontal connections, and a full connection.

\begin{table}
$$\xymatrix{ & & T(E) \ar[drr]^{\<T(q),p\>} \\ E \ar[rrrr]_{\< q0,q0_{\sf q}\>} \ar[urr]^\lambda & & & & T(M) \x_M E}$$
	
\begin{center}
\begin{tabular}{|l|l|} \hline
$\lambda: E \to T(E)$ &   {\bf vertical lift}  \\
$k: T(E) \to E$  &      {\bf vertical descent}: $\lambda k=1$ and $kq = T(q)p$ \\
$K: T(E) \to E$ &      {\bf vertical connection}: \\ & ~~~~~~~~~~$\lambda K = 1$, $Kq = T(q)p$, \\
  &  ~~~~~~~~~~and $T(\lambda) c T(K) = K\lambda = \ell T(K)$   \\ \hline
$U:=\< T(q),p \>: T(E) \to T(M) \x_M E$ &  {\bf horizontal descent}  \\
$h: T(M) \x_M E \to T(E)$            &     {\bf horizontal lift}:    $h\<T(p),p\> = 1$ \\
$H: T(M) \x_M E \to T(E)$            &   {\bf horizontal connection}:  \\ & ~~~~~~~~~~$H\<T(q),p\> = 1$, \\ &  ~~~~~~~~~~$H\ell = (\ell \x 0)T(H)$, and $HT(\lambda)c= (0 \x \lambda)T(H)$ \\ \hline
$(K,H)$    &   {\bf connection}: \\ 
& ~~~~~~~~~~$\< K,p\>\nu + \<T(p),p\>H = 1$ and $HK = \pi_0q\zq$ \\ 
&  ~~~~~~~~~~where $H$ is a horizontal connection \\ &  ~~~~~~~~~~and  $K$ is a vertical connection \\ \hline
\end{tabular}
\caption{Data for connections}
\label{dataforConn}
\end{center}
\end{table}

\medskip
	
The canonical example of a connection is on a differential object:

\begin{proposition}\label{diffObjectConnection}
Any differential object $A$ (seen as a differential bundle over $1$) has a canonical connection $(K,H)$ where $K = \hat{p}$ and $H = \pi_10$.  
\end{proposition}
\begin{proof}
By Proposition \ref{diffObjectVD}, $\hat{p}$ is a vertical connection, and by Proposition \ref{diffObjectHL}, $\pi_1 0$ is a horizontal connection.  By Proposition 3.4 of \cite{diffBundles}, $\pi_1 0 \hat{p} = \pi_1 ! \zq$ so the $0$ requirement is satisfied.   

For the summation requirement of a connection, we need to show that 
	\[ \<\<\hat{p}\lambda,p0\>T(\pq),\<!,p\>\pi_1 0\>+ = 1. \]
But $\<!,p\>\pi_10 = p0$, so using $+$ on that term does nothing, so the expression reduces to
	\[ \<\hat{p}\lambda,p0\>T(\pq). \]
Since the codomain of this map is $T(E)$, a product, it suffices to check how it is affected by the two projections on $T(E)$, $\hat{p}$ and $p$.  By additivity of $\hat{p}$,
	\[ \<\hat{p}\lambda,p0\>T(\pq)\hat{p} = \<\hat{p}\lambda\hat{p},p0\hat{p}\>\pq = \<\hat{p},!\zq\>\pq = \hat{p}, \]
and by naturality of $p$,
	\[ \<\hat{p}\lambda,p0\>T(\pq)p = \<\hat{p}\lambda p, p0p\>\pq = \<!\zq,p\>\pq = p. \]
Thus $\<\hat{p}\lambda,p0\>T(\pq)$ is indeed the identity on $T(E)$, as required. 
\end{proof}

%\begin{proposition}\label{exCanonicalAffine}
%If $A \times M \to^{\pi_1} M$ is a trivial differential bundle, it has a canonical connection with connector
%	\[ T(A \times M) \to^{\<T(\pi_0)\hat{p},T(\pi_1)p\>} A \times M \]
%and horizontal connection
%	\[ T(M) \times A \times M \to^{\<\pi_10,\pi_0\>} T(A \times M). \]
%\end{proposition}
%Straightforward to check, recalling that $\hat{p} = \{1\}$ and the lift is 
%	\[ A \times M \to^{\<\pi_0 \lambda, \pi_10\>} T(A \times M) \]

%%%%%%%%%%%%%%%%%%%%%%%%%%%%%%%%%%%%%%%%%%%%%%%%%%%%%%%%%%

\subsection{Basic properties of connections}

Applying the tangent bundle functor to a connection gives a connection:

\begin{proposition}\label{tOfaConnection}
If $(K,H)$ is a connection on $\sf q$, then $(cT(K),(c \times 1)T(H)c)$ is a connection on $T({\sf q})$.
\end{proposition}
\begin{proof}
By Proposition \ref{tVertDescent}, $cT(K)$ is a vertical connection, and by \ref{tHorizontalLift}, $(c \times 1)T(H)c$ is a horizontal connection.  Thus all that remains to show is that the pair satisfies the equations to be a connection:
\begin{eqnarray*}
&   & (c \times 1)T(H)ccT(K) \\
& = & (c \times 1)T(H K) \\
& = & (c \times 1)T(\pi_1 q \zq) \\
& = & \pi_1 T(q) T(\zq) 
\end{eqnarray*}
so that the $0$ condition holds.  For the other desired equation, we need to show that
\begin{eqnarray*}
&   & \<cT(K),p\>\<\pi_0T(\lambda)c,\pi_10\>T^2(\pq) + \<T^2(q),p\>(c \times 1)T(H)c \\
& = & \<cT(K \lambda)c,p0\>T^2(\pq) + \<T^2(q)c,p\>T(H)c 
\end{eqnarray*}
is equal to 1.  However, we know that since the pair $(K,H)$ forms a connection,
	\[ \<K,p\>v + \<T(q),p\>H = \<K\lambda,p0\>T(\pq) + \<T(q),p\>H = 1 \]
Applying $T$ to both sides of this equation gives
	\[ [\<T(K\lambda),T(p)T(0)\>T^2(\pq)] \  T(+) \ [\<T^2(q),T(p)\>T(H)] = 1 \]
Using proposition 2.4(v) in \cite{sman3}, we can change this into
	\[ [\<T(K\lambda),T(p)T(0)\>T^2(\pq)c + \<T^2(q),T(p)\>T(H)c]c = 1 \]
Then applying $c$ to the left and right sides of both equations and using the fact that $c^2 = 1$, we get
\begin{eqnarray*}
c\<T(K\lambda),T(p)T(0)\>T^2(\pq)c + c\<T^2(q),T(p)\>T(H)c & = & 1 \\
\<cT(K\lambda),cT(p)T(0)\>cT^2(\pq) + \<cT^2(q)c,cT(p)\>T(H)c & = & 1 \mbox{ (by naturality of $c$)} \\
\<cT(K\lambda)c,pT(0)c\>T^2(\pq) + \<T^2(q)c,p\>T(H)c & = & 1 \mbox{ (naturality and coherence for $c$)} \\
\<cT(K\lambda)c,p0\>T^2(\pq) + \<T^2(q)c,p\>T(H)c & = & 1 \mbox{ (coherence for $c$)} 
\end{eqnarray*}
so that the identity condition holds.   
\end{proof}

Pullbacks of connections are connections:
\begin{proposition}
If $(K,H)$ is a connection on $\sf q$, then $(f^*K,f^*H)$ is a connection on $f^{*}({\sf q})$.
\end{proposition}
\begin{proof}
Proposition \ref{pullbackVD} defines $f^*K = K \times p$ and shows that it is a vertical connection.    Proposition \ref{pullbackHL} defines $f^*H = \<(T(f) \times f^*_E)H,\pi_0\>$ and shows that it is a horizontal connection.  It remains to show the connection equations for this pair.  For the 0 condition,
\begin{eqnarray*}
&   & f^*Hf^*K \\
& = & \<(T(f) \times f^*_E)H,\pi_0\>(K \times p) \\
& = & \<(T(f) \times f^*_E)HK, \pi_0 p \> \\
& = & \<(T(f) \times f^*_E)\pi_1q \zq, \pi_0p\> \mbox{ ($0$ condition for $(K,H)$)} \\
& = & \<\pi_1 f^*_E q \zq,  \pi_1f^*(q)\> \mbox{ (pullback definition in second component)} \\
& = & \<\pi_1 f^*(q) f \zq,  \pi_1f^*(q)\> \mbox{ (pullback definition in first component)} \\
& = & \pi_1 f^*(q) \<f \zq, 1\> 
\end{eqnarray*}
as required, since the zero for $f^*(q)$ is $\<f\zq,1\>$.  

For the summation condition, we need to show that 
	\[ \<\<f^*Kf^*\lambda,p0\>T(f^*\pq),Uf^*H\>+ = 1_{T(f^*E)} \ (\dagger) \] 
We'll first expand $\<f^*Kf^*\lambda,p0\>T(f^*\pq)$:
\begin{eqnarray*}
&   & \<f^*Kf^*\lambda,p0\>T(f^*\pq) \\
& = & \<(K \times p)(\lambda \times 0),p0\>T(\<\<\pi_0f^*_E,\pi_1f^*_E\>\pq,\pi_1f^*(q)\>) \\
& = & \<K \lambda \times p0,p0\>\<\<T(\pi_0)T(f^*_E),T(\pi_1)T(f^*_E)\>T(\pq),T(\pi_1)T(f^*(q))\> \\
& = & \<\<(K \lambda \times p0)T(f^*_E),p0T(f^*_E)\>T(\pq),p0T(f^*(q))\> \\
& = & \<\<T(f^*_E)K \lambda,T(f^*_E)p0\>T(\pq),T(f^*(q))p0\> \mbox{ ($T(f^*_E)$ a projection, naturality of $p0$)} \\ 
& = & \<T(f^*_E) \<K \lambda,p0\>T(\pq),T(f^*(q))p0\>
\end{eqnarray*}
Now we expand $Uf^*H$:
\begin{eqnarray*}
&   & Uf^*H \\
& = & \<T(f^*(q)),p\>\<(Tf \times f^*_E)H,\pi_0\> \\
& = & \<T(f^*(q))T(f),pf^*_E\>H,T(f^*(q))\> \\
& = & \<T(f^*_E)T(q),T(f^*_E)p\>H,T(f^*(q))\> \mbox{ (pullback definition, naturality of $p$)}  \\
& = & \<T(f^*_E)\<T(q),p\>H,T(f^*(q))\> 
\end{eqnarray*}
Then the expression on the left side of $\dagger$ is
\begin{eqnarray*}
&   & \<\<T(f^*_E) \<K \lambda,p0\>T(\pq),T(f^*(q))p0\>, \<T(f^*_E)\<T(q),p\>H,T(f^*(q))\> \>+ \\
& = & \<\<T(f^*_E) \<K \lambda,p0\>T(\pq),\<T(f^*_E)\<T(q),p\>H\>+,\<T(f^*(q)p0,T(f^*(q)\>+\> \mbox{ (naturality of $+$)} \\
& = & \<T(f^*_E)\<\<K\lambda,p0\>T(\pq),\<T(q),p\>H\>+,T(f^*(q))\<p0,1\>+\> \\
& = & \<T(f^*_E),T(f^*(q))\> \mbox{ (connection definition for $(K,H)$ and addition of $0$)}\\
& = & 1_{T(f^*E)}
\end{eqnarray*}
as required.  

\end{proof}

\begin{example}\label{canonicalAffineConnection}{\em ~
Using \ref{diffObjectConnection}, any trivial differential bundle $A \times M \to^{\pi_1} M$ (for $A$ a differential object) has a connection.  In particular, the tangent bundle of any differential object $A$ has a connection given by
	\[ K = \<T(\hat{p})\hat{p},pp\> \mbox{ and } H = \<!\zq,\pi_0\hat{p},\pi_1\hat{p},\pi_1p\>. \]
(see Example \ref{exampleTrivialDiffBundleVD} and Example \ref{canonicalAffineHL}).  
}
\end{example}

Note that while one can retract both a vertical connection and a horizontal connection (Propositions \ref{retractVD} and \ref{retractHL}), unfortunately, retracting a connection $(K,H)$  \emph{does not} necessarily produce a connection.  The problem is the zero condition: for this, one needs to consider
	\[ (T(s_0) \times s_1)HT(r_1)T(s_1)Kr_1 \]
and the idempotent $r_1s_1$ prevents bringing the $HK$ together.  
	
An important aspect of a connection on $q: E \to M$ is that it provides a decomposition of the tangent bundle of $E$:  	

\begin{proposition}\label{connectionGivesPullback}
If $\sf q$ has a connection $(K,H)$, then 
\[
\bfig
	\node a(-100,0)[T(E)]
	\node b(500,250)[E]
	\node c(500,0)[T(M)]
	\node d(500,-250)[E]
	\node e(1100,0)[M]
	
	\arrow[a`b;K]
	\arrow|m|[a`c;T(q)]
	\arrow|b|[a`d;p]
	\arrow[b`e;q]
	\arrow|m|[c`e;p]
	\arrow|b|[d`e;q]
\efig
\]
is a limit diagram.
\end{proposition}
\begin{proof}
Since $(K,q)$ is a bundle morphism, $Kq = pq = T(q)p$, so the above diagram commutes.  Suppose now we have maps $f_0: X \to E, f_1: X \to T(M), f_2: X \to E$ with $f_0q = f_1p = f_2q$.  We claim that 
	\[ (f_0\lambda \pqone f_20) + \<f_1,f_2\>H \]
is the required unique map to $T(E)$. For the first projection, we have
\begin{eqnarray*}
&   & [(f_0\lambda \pqone f_20) + \<f_1,f_2\>H]K \\
& = & f_0\lambda K \pq f_2 0 K \pq \<f_1,f_2\>H K \mbox{ ($K$ is additive for both additions on $T(E)$)} \\
& = & f_0 + f_2 q \zq + f_2  q \zq \mbox{ ($K$ is a connector and $H$ a horizontal connection)} \\
& = & f_0 + f_0 q \zq + f_0 q \zq \mbox{ (by assumption on $f_2$)} \\
& = & f_0
\end{eqnarray*}
For the second projection,
\begin{eqnarray*}
&   & [(f_0\lambda \pqone f_20) + \<f_1,f_2\>H]T(q) \\
& = & (f_0\lambda \pqone f_20)T(q) + \<f_1,f_2\>H T(q) \mbox{ ($T(f)$ is additive for any $f$)} \\
& = & f_2 0 T(q) + \<f_1,f_2\>\pi_0 \mbox{ ($\pq q = \pi_1 q$ and $H$ a horizontal connection)} \\
& = & f_2 q 0 + f_1 \\
& = & f_1 p0 + f_1 \mbox{ (by assumption on $f_1$)} \\
& = & f_1
\end{eqnarray*}
For the third projection,
\begin{eqnarray*}
&   & [(f_0\lambda \pqone f_20) + \<f_1,f_2\>H]p \\
& = & \<f_1,f_2\>H p \\
& = & \<f_1,f_2\>\pi_1 \\
& = & f_2
\end{eqnarray*}
Finally, if we have some other $h: X \to T(E)$ with $hK = f_0, hT(q) = f_1, hp = f_2$, then
	\[ h = h[(k\lambda \pqone p0) + \<Tq,p\>H] = (f_0\lambda \pqone f_20) + \<f_1,f_2\>H \]
as required.
\end{proof}

\subsection{Remark}
Rory Lucyshyn-Wright has noticed the significance of this pullback: it makes apparent that $T(E)$ is actually a fibred biproduct of $E, T(M)$, and  $E$.  This has led to his development of a different characterization of connections in tangent categories in terms of biproducts: see \cite{rory} for further details. \\

Another important property of a connection is that each part of the connection uniquely determines the other part.   
\begin{proposition}\label{propConnectDeterm}
Let $\sf q$ be a differential bundle.  
\begin{itemize}
	\item If $H$ is a horizontal connection on $\sf q$ for which $(K_1,H)$ and $(K_2,H)$ are both connections on $\sf q$, then $K_1=K_2$. 
	\item If $K$ is a vertical connection on $\sf q$ for which $(K,H_1)$ and $(K,H_2)$ are connections on $\sf q$, then $H_1 = H_2$.  
\end{itemize}
\end{proposition}
\begin{proof}
For the first part,
\begin{eqnarray*}
\<K_1,p\>\mu + UH & = & 1 \\
\<K_1,p\>\mu K_2 + UH K_2 & = & K_2 \\
\<K_1,p\>\pi_0 + 0 & = & K_2 \mbox{ (by lemma \ref{lemmaVConnector})} \\
K_1 & = & K_2
\end{eqnarray*}
For the second part,
\begin{eqnarray*}
\<K,p\> \mu + UH_1 & = & 1 \\
H_2\<K,p\> \mu + H_2UH_1 & = & H_2 \\
\<H_2K,H_2p\> \mu + H_1 & = & H_2 \\
\<p0,1\> \mu + H_1 & = & H_2 \\
0 + H_1 & = & H_2 \\
H_1 & = & H_2
\end{eqnarray*}
\end{proof}

%%%%%%%%%%%%%%%%%%%%%%%%%%%%%%%%%%%%%%%%%%%%%%%%%%%%%%%%%%%%%%%%%%%%

\subsection{Connections defined by horizontal or vertical connections}

We now turn to the question of whether the existence of a vertical or horizontal connection is enough to define a full connection.   

  \begin{proposition}\label{horizontalToConnection}
Let $(\X,\T)$ be a tangent category with negatives, and $H$ a horizontal connection on a differential bundle $\sf q$.  Then the pair $(\{1-UH\},H)$ is a connection on $\sf q$.
\end{proposition}
\begin{proof}\footnote{Some ideas for this proof are drawn from \cite[Proposition 1]{bungeConnections}.}
We first have to check $1-UH$ equalizes (to ensure that $\{1-UH\}$ is well-defined).  Indeed, since $T(q)$ is additive,
	\[ (1 - UH)T(q) = T(q) - UH T(q) = T(q) - U\pi_0 = T(q) - T(q), \]
so $(1-UH)T(q) = (1-UH)T(q)p0$, as required.  

To prove that $K = \{1-UH\}$ is a vertical connection, we first need $\lambda K = 1$.  Indeed, using various parts of \cite[Lemma 2.12]{diffBundles},
\begin{eqnarray*}
\lambda K & = & \lambda \{1 - UH\} \\
& = & \{\lambda - \lambda U H \} \\
& = & \{\lambda - \<0,0\> H \}  \\
& = & \{\lambda - 0\} \mbox{ ($H$ is additive)} \\
& = & \{\lambda\} \\
& = & 1
\end{eqnarray*}
Moreover, using Propositions \ref{UisLinear}, \ref{constructingLinearMorphisms}, \ref{propBracketLinear1}, and \ref{propBracketLinear2}, $(\{1-UH\},p)$ is a linear bundle morphism from $T({\sf q})$ to $\sf q$ and $(\{1-UH\},q)$ is linear from $\sf{p_E}$ to $\sf{q}$.  Thus $\{1-UH\}$ is a vertical connection.

Finally, we calculate the two identities needed for a connection:
\[ HK = H\{1 - UH\} = \{H - HUH\} = \{H - H\} = \{Hp0\} = \{\pi_0 0 \} = \pi_0 \{0\} = \pi_0 q \zq \]
(using \cite[Lemma 2.12.iv]{diffBundles} in the last equality), and 
	\[ \<K,p\>\mu + UH = \<\{1-UH\},p\>\mu + UH = (1-UH)_{|\mu}\mu + UH = 1-UH + UH = 1 \]
as required.  
\end{proof}

For the converse direction, we again need to assume negatives, but we also need to make an additional assumption, namely that the differential bundle in question has at least one horizontal connection.  This is true, for example, of any vector bundle in the category of smooth manifolds (for example, see \cite[pg. 132]{diffGeoDynamicalSystems}).  

\begin{proposition}\label{propConstructingHorizLift}
Let $(\X,\T)$ be a tangent category with negatives, ${\sf q}$ a differential bundle, and $K$ a vertical connection on ${\sf q}$. If ${\sf q}$ has a horizontal connection $J$, then $J(1-\<K,p\>\mu)$ is a horizontal connection on ${\sf q}$, and the pair $(K,J(1-\<K,p\>\mu))$ is a connection. 
\end{proposition}
\begin{proof}
By Theorem \ref{thmFinslerEquivalence}, $\<K,p\>$ is a Finsler descent and hence linear for both linear structures on $T(E)$.  $\mu$ is also linear for both structures by Lemma \ref{muLinearity} and Proposition \ref{examplesLinearMorphisms}{\em (x)} (since by the definition of a tangent category, it is part of a pullback).  Thus by Proposition \ref{examplesLinearMorphisms}, $1 - \<K,p\>\mu$ is as well, and since $J$ is linear in both components by assumption, $J(1-\<K,p\>\mu)$ is as well.  

$J(1-\<K,p\>\mu)$ is a section of $U$ since
	\[ J(1-\<K,p\>\mu)T(q) = JT(q) - J\<K,p\>\mu T(q) = \pi_0 - \pi_0 q0 = \pi_0, \]
and
	\[ J(1-\<K,p\>\mu)p = Jp = \pi_1, \]
so $J(1-\<K,p\>\mu)$ is a horizontal connection.  

We now need to show it satisfies the required identities to form a connection with $K$.  For one of the equations:
	\[ J(1-\<K,p\>\mu)K = J(K - \<K,p\>\mu K) = J(K - K) = J pq \zq = \pi_1 q \zq \]
as required.  The other equation we need to show is that
	\[ UJ(1 - \<K,p\>\mu) + \<K,p\>\mu = 1. \]
Since $J$ is a horizontal connection, by the previous proposition (\ref{horizontalToConnection}), it is part of a connection with vertical connection $K' = \{1 - UJ\}$.  Thus by proposition \ref{connectionGivesPullback}, $T(E)$ is a pullback with projections $K', T(q)$, and $p$, and so it suffices to check the above identity when post-composing by those maps.  For $T(q)$:
\begin{eqnarray*}
& & (UJ(1 - \<K,p\>\mu) + \<K,p\>\mu)T(q) \\
& = & UJT(q) - UJ\<K,p\>\mu T(q) + \<K,p\>\mu T(q) \\
& = & T(q) - T(q)p0 + T(q)p0 \\
& = & T(q)
\end{eqnarray*}
for composing with $p$, any component in the sum gives the same result, so 
\[ (UJ(1 - \<K,p\>\mu) + \<K,p\>\mu)p = UJp = p, \]
and for $K'$,
\begin{eqnarray*}
&   & (UJ(1 - \<K,p\>\mu) + \<K,p\>\mu)K' \\
& = & UJK' - UJ\<K,p\>\mu K' + \<K,p\>\mu K' \\
& = & \<T(q),p\>\pi_1q\zq - UJ\<K,p\>\pi_0 + \<K,p\>\pi_0 \\
& = & pq\zq -UJK + K \\
& = & (1 - UJ)K \\
& = & \{1 - UJ\} \mbox{ (by lemma \ref{lemmaKUniversality})} \\
& = & K'
\end{eqnarray*}
so the required identities have been established.  
\end{proof}
Note that by proposition \ref{propConnectDeterm}, any other horizontal connection $J'$ would give the same connection.  

To sum up the results of this section: with negatives, giving a horizontal connection is sufficient to give a connection.  With negatives and the existence of at least one horizontal connection on a bundle, giving a vertical connection is sufficient to give a connection.  In a tangent category which lacks negatives, however, neither of these results may hold, and this is why we have defined a connection as having both a vertical connection and a horizontal connection.  

%%%%%%%%%%%%%%%%%%%%%%%%%%%%%%%%%%%%%%%%%%%%%%%%%%%%%%%%%%%%%%%%%%%%%%

\subsection{Parallel transport}\label{secParallelTransport}

%%%%%%%%%%%%%%%%%%%%%%%%%%%%%%%%%%%%%%%%%%%%%%%%%%%%%%%%%%%%%%%%%%%%%%

In smooth manifolds, a connection on a vector bundle $q: E \to M$ has an associated notion of \emph{parallel transport}.  That is, given some curve $\gamma: I \to M$ (where $I$ is an interval in $\mathbb{R}$) and a point $e_0 \in E$ with $q(e_0) = \gamma(0)$, there exists a {\em unique\/} curve, $\gamma': I \to E$ which is ``above $\gamma$'', in the sense that $\gamma'q = \gamma$, with its time derivative with respect to the connection being $0$. 
The construction of parallel transport for smooth manifolds relies on being able to solve certain dynamical systems (i.e. certain ordinary differential equations with given initial conditions).  Clearly this is not something that will be 
possible in an arbitrary tangent category.  

To solve dynamical systems in an arbitrary Cartesian tangent category one needs ``curve'' objects: these are preinitial dynamical systems.  With the assumption that one has a curve object one 
can show that connections do indeed give rise to parallel transport.  We start this programme by describing dynamical systems:

\begin{definition}
In a Cartesian tangent category, a \textbf{dynamical system} consists of a triple $(M,x_0,x_1)$ where $M$ is an object, $x_1: M \to T(M)$  is a vector field on $M$ (so that $x_1p = 1_M$) ,and $x_0: 1 \to M$ is a point of $M$.   A {\bf morphism} of dynamical systems $f: (M,x_0,x_1) \to (N,y_0,y_1)$ is a map $f: M \to N$ such that:
$$\xymatrix{1 \ar[r]^{x_0} \ar[dr]_{y_0} & M \ar[d]^f  \ar[rr]^{x_1} & & T(M) \ar[d]^{T(f)}  \\ & N \ar[rr]_{y_1} && T(N)}$$
\end{definition}

Observe that while dynamical systems with their morphisms clearly form a category, more is true:

\begin{proposition}
The dynamical systems of a Cartesian tangent category, $\X$, with their morphisms form a Cartesian tangent category, ${\sf Dyn}(\X)$.
\end{proposition}  

\proof The tangent functor is given by $T(M,x_0,x_1) := (T(M),T(x_0),T(x_1)c)$ (recall $T(1)$ is a final object as the category in Cartesian and so, without loss of generality, we may assume $1=T(1)$) 
and sends a morphism $f:(M,x_0,x_1) \to (N,y_0,y_1)$ to $T(f)$.  $T(f)$ is a morphism of dynamical systems since
$$T(f)T(y_1)c = T(fy_1)c = T(x_1T(f))c =T(x_1)cT^2(f)$$
and is functorial since $T$ is.  

Next we verify that we can define the projection, zero, addition, lift, and canonical flip pointwise:
\begin{eqnarray*} 
p & := & p: (T(M),T(x_0),T(x_1)c) \to (M,x_0,x_1) \\
0 &:= & 0: (M,x_0,x_1) \to (T(M),T(x_0),T(x_1)c) \\
+ & := &  +: (T_2(M),T_2(x_0),T_2(x_1) (c \x_M c)) \to (T(M),T(x_0),T(x_1)c) \\
{\sf l} & := & \ell: (T(M),T(x_0),T(x_1)c) \to (T^2(M),T^2(x_0),T(T(x_1)c)c) \\
{\sf c} & := & c: (T^2(M),T^2(x_0),T(T(x_1)c)c) \to (T^2(M),T^2(x_0),T(T(x_1)c)c) 
\end{eqnarray*}
For example, to show that the pointwise defined lift is well-defined we use the following commuting diagram:
$$\xymatrix{1 \ar[rr]^{T(x_0)} \ar[drr]_{T^2(x_0)} & & T(M) \ar[rr]^{T(x_1)} \ar[d]_{\ell} & & T^2(M) \ar[d]^{\ell} \ar[rr]^c && T^2(M) \ar[d]^{T(\ell)} \\
                     & & T^2(M) \ar[rr]_{T^2(x_1)} & & T^3(M) \ar[rr]_{T(c)c} & & T^3(M)}$$
and to check the others are well-define is done in a similar style.   Finally, because the structure is defined pointwise all the coherence and limit requirements are inherited.
\endproof

The vector field $x_1$ of a dynamical system $(M,x_0,x_1)$ may be regarded as determining a ``differential equation'' where $x_1$ indicates the desired derivative and $x_0$ determines  
the ``initial conditions''.   A solution can then be viewed as a certain homomorphism of dynamical systems from a special fixed dynamical system $(C,c_0,c_1)$ to the specified system $(M,x_0,x_1)$.  

For example, in the category of smooth manifolds, consider $C = \mathbb{R}$ with $c_0 = 0$, and $c_1 = \<i,1_{\mathbb{R}}\>$, where $i$ is the multiplicative unit in $\mathbb{R}$.  Suppose also that we have an arbitrary dynamical system $(M,x_0,x_1)$; for simplicity let us consider the case when $M = \mathbb{R}$ also.  Note that in this case, since $T(M) = T(\mathbb{R}) = \mathbb{R} \times \mathbb{R}$, then to give a vector field $x_1: M \to T(M)$ is equivalent to simply giving a smooth map $f: \mathbb{R} \to \mathbb{R}$.  Then a homomorphism of these dynamical systems consists of a map
	\[ \gamma: (C,c_0,c_1) \to (M,x_0,x_1) \]
with the properties that $\gamma(c_0) = x_0$ and $c_1T(\gamma) = \gamma x_1$.  But for this second condition, since these maps are into a product $(\mathbb{R} \times \mathbb{R})$ their equality is determined by their equality on each component.  Since	
	\[ c_1T(\gamma)p = c_1p\gamma = \gamma = \gamma x_1  \]
their equality on the second component (which has projection $p$) is guaranteed.  The equality on the second component says that
	\[ c_1T(\gamma)\hat{p} = \gamma x_1 \hat{p} = \gamma f \]
But the term $c_1T(\gamma)\hat{p}$ is simply the derivative of $\gamma$, so in ordinary calculus terms this condition simply asks that for any $x \in \mathbb{R}$
	\[ \gamma'(x) = f(\gamma(x)). \]
In other words, asking for a homomorphism of of these dynamical systems is equivalent to asking for a solution to the first order ordinary differential equation $\gamma' = f(\gamma)$, $\gamma(0) = x_0$.  

It is also well-known that solutions to ordinary differential equations are unique \cite[Theorem IV.1.3]{lang}.  Thus, in an arbitrary tangent category we may hope to ask for the following:

%In the category of smooth manifolds a solution is a curve in $M$, that is a smooth map with domain an interval $\gamma: (a,b) \to M$, such that $\gamma(0) = x_0 and $(\gamma'(z),z) = x_1(z)$ for $z \in (a,b)$.  Our objective is to describe the formal analogue of this ability to solve differential equations in smooth manifolds.  Our first attempt, is  to suppose ``solutions'' for these differential equations are provided by having an initial object in  ${\sf Dyn}(\X)$:

\begin{definition} \label{total-curve-object}
A \textbf{total curve object}, $(C,c_0,c_1)$ is a dynamical system which is initial in the category of dynamical systems.   That is, for each dynamical system $(M,x_0,x_1)$, there is a unique 
homomorphism $\gamma$ such that:
$$\xymatrix{1 \ar[r]^{c_0} \ar[dr]_{x_0} & C \ar@{..>}[d]^\gamma  \ar[rr]^{c_1} & & T(C) \ar@{..>}[d]^{T(\gamma)}  \\ & M \ar[rr]_{x_1} && T(M)}$$
\end{definition}

As described above, we then think of the unique map $\gamma: C \to M$ as a solution to the differential equation determined by the dynamical system.   An example of such a total curve object 
is given by $(D_\infty,0, x \mapsto \lambda d.d+x)$ in the category of microlinear spaces of a model of SDG  (see \cite[Theorem 2.4]{kockReyesSolutionsODEs}). 

However, having a total curve object is often too much to expect in less powerful tangent categories: such an object does not exist in the category of smooth manifolds!  In particular, the dynamical system $(\mathbb{R},c_0,c_1)$ as described above is not a total curve object.  While any two solutions that exist are unique, an arbitrary dynamical system/first order differential equation need not have a total solution.  For example, on $\mathbb{R}$, taking the vector field $x_1(x) = \<-x^2,x\>$ with initial condition $\gamma(0) = 1$ has no total solution $\gamma: \mathbb{R} \to \mathbb{R}$ (eg., see \cite[pgs. 137--138]{spivak1}) -- although it does have a {\em partial\/} solution $\gamma(x) = \frac{-1}{x-1}$.

What is true, however, is that \emph{linear} systems always have total solutions \cite[Proposition IV.1.9]{lang}.  Let us describe linear systems in our setting:

%for example, the category of smooth manifolds does not have such an object.  However, in smooth manifolds the intervals of the real line with the constantly unit vector field: $$c_1:(a,b) \to T((a,b)) = \mathbb{R} \x (a,b); z \mapsto \<1,z\>$$ are preinitial -- in the sense that there is at most one homomorphism from this interval dynamical system to any other object. This suggests that we might usefully demand less than an initial object and yet obtain a more useful notion of ``solution'' for differential equations as formulated by dynamical systems. 

\begin{definition} Say that a dynamical system $(M,x_0,x_1)$ is \textbf{linear} in case the vector field $x_1$ is part of a 
linear morphism of differential bundles.  That is, there is a differential bundle $m: M \to B$ and a map $b: B \to TM$ such that the following diagrams commute:
$$\xymatrix{& T(M) \ar[rr]^{T(x_1)} & & T^2(M) \\ M \ar[ur]^\lambda \ar[d]_m \ar[rr]^{x_1} & & T(M) \ar[ur]_{T(\lambda) c}  \ar[d]^{T(m)} \\ B \ar[rr]_b & & T(B)}$$
\end{definition}
This allows us the following useful weakening of the notion of a total curve object:

\begin{definition} \label{curve-object}
A dynamical system $(C,c_0,c_1)$ is a {\bf curve object} in case 
\begin{itemize}
\item It is preinitial in the category of dynamical systems (that is, any two homorphisms from it are equal);
\item It has homorphisms to all linear dynamical systems.  
%\item For each differential bundle $q: E \to M$ and curve $\gamma: C \to M$ the pullback 
%$$\xymatrix{E \x_M C \ar[d]^{\pi_1} \ar[rr]^{\pi_0} && E \ar[d]^{q} \\ C \ar[rr]_{\gamma} && M}$$
%exists and is preserved by all powers of $T$.   
\end{itemize}
\end{definition}

Thus, in the category of smooth manifolds, $\mathbb{R}$ with $c_0 = 0$ and $c_1 = \<i,1_{\mathbb{R}}\>$ is such a curve object; in fact, more generally any interval $C = (a,b)$ around $0$ in $\mathbb{R}$ with $c_0 = 0$ and $c_1 = \<i,1_{C}\>$ is an example of such an object.

With an assumption of such an object, we can now show completely formally how connections give rise to parallel transport.

\begin{theorem}(Parallel transport)
Suppose $(C,c_1,c_0)$ is a curve object in a Cartesian tangent category, and $\sf q$ is a differential bundle with connection $(K,H)$.  Suppose that $\gamma: C \to M$ is a map such that the pullback of $q$ along $\gamma$ exists and is preserved by $T$, and $e_0: 1 \to E$ is a point of $E$ such that $e_0q = c_0 \gamma$.  Then there exists a unique map $\hat{\gamma}: C \to E$ such that:
\begin{enumerate}[(i)]
	\item \textbf{$\hat{\gamma}$ starts at $e_0$}: that is, $c_0\hat{\gamma} = e_0$;
	\item \textbf{$\hat{\gamma}$ is above $\gamma$}: that is, $\hat{\gamma}q = \gamma$;
	\item \textbf{$\hat{\gamma}$ is parallel}: that is, $\nabla_K(c_1,\hat{\gamma}) = \gamma 0_{\sf q}$, i.e., the following diagram commutes:
	\[
	\bfig
		\node A(0,0)[C]
		\node B(250,250)[T(C)]
		\node C(800,250)[T(E)]
		\node E(1050,0)[E]
		\node M(525,-250)[M]
		\arrow[A`B;c_1]
		\arrow[B`C;T(\hat{\gamma})]
		\arrow[C`E;K]
		\arrow[A`M;\gamma]
		\arrow[M`E;0_{\sf q}]
	\efig
	\]
	
\end{enumerate}
\end{theorem}
\begin{proof}
The key idea is to use the universal property of $C$ not on the object $E$, but on the pullback of $E$ along $\gamma$:
\[
\bfig
	\square<500,350>[C \times_M E`E`C`M;\pi_1`\pi_0`q`\gamma]
\efig
\]
in other words, the pullback bundle $\gamma^*({\sf q})$.  We define a map $F: C \times_M E \to T(C \times_M E)$ by 
	\[ F := \<\pi_0c_1,\<\pi_0c_1T(\gamma),\pi_1\>H\>. \]
We first claim that this is a well-defined linear vector field on $\gamma^*({\sf q})$.  First, note that $\<\pi_0c_1T(\gamma),\pi_1\>$ is well-defined since
	\[ \pi_0 c_1T(\gamma)p = \pi_0 c_1p \gamma = \pi_0 \gamma = \pi_1 q \]
by definition of the pullback $C \times_M E$.  

Next, we need to show that $F$ itself is well-defined.  $F$ is going into $T(C \times_M E)$, so we need to check the two components of $F$ are equal when post-composed by $T(q)$ and $T(\gamma)$:
	\[ \<\pi_0 c_1T(\gamma),\pi_1\>HT(q) = \<\pi_0c_1T(\gamma),\pi_1\>\pi_0 = \pi_0c_1T(\gamma) \]
as required.

Next, we want to show that $F$ is a vector field.  Indeed,
\begin{eqnarray*}
Fp & = & \<\pi_0c_1,\<\pi_0c_1T(\gamma),\pi_1\>H\>p \\
& = & \<\pi_0c_1p,\<\pi_0c_1T(\gamma),\pi_1\>Hp\> \mbox{ (naturality of $p$)} \\
& = & \<\pi_0,\<\pi_0c_1T(\gamma),\pi_1\>\pi_1\> \mbox{ ($c_1$ vector field, definition of $H$)} \\
& = & \<\pi_0,\pi_1\> \\
& = & 1_{C \times_M E} 
\end{eqnarray*}
as required.  

Finally, we want to show that $F$ is linear as a map between the differential bundles $\gamma^*({\sf q})$ and $T(\gamma^*({\sf q}))$.  The base map $b$ we can clearly take as $\<\pi_0c_1T(\gamma),\pi_1\>H$.  For the linearity condition we need to show that 
	\[ F T(0 \times \lambda)c = (0 \times \lambda)T(F). \]
Thus, consider
\begin{eqnarray*}
FT(0 \times \lambda)c 
& = & \<\pi_0c_1,\pi_0c_1T(\gamma),\pi_1\>H\> (T(0)c \times T(\lambda)c) \mbox{ (naturality of $c$)}\\
& = & \<\pi_0c_1,\pi_0c_1T(\gamma),\pi_1\>H\>(0 \times T(\lambda)c) \mbox{ (coherence of $c$)} \\
& = & \<\pi_0c_10,\pi_0c_1T(\gamma),\pi_1\>HT(\lambda)c\> \\
& = & \<\pi_0c_10,\pi_0c_1T(\gamma),\pi_1\>(0 \times \lambda)T(H)\> \mbox{ (linearity of $H$)} \\
& = & \<\pi_00T(c_1),\pi_00T(c_1)T^2(\gamma),\pi_1\lambda\>T(H)\> \mbox{ (naturality of $0$)} \\
& = & (0 \times \lambda)\<T(\pi_0)T(c_1),\<T(\pi_0)T(c_1)T^2(\gamma),T(\pi_1)\>T(H)\> \\
& = & (0 \times \lambda)T(F)
\end{eqnarray*}
as required.  

Thus, by the universal property of the curve object $C$, there exists a unique map $\beta: C \to ~C \times_M E$ such that the diagrams
\[
\bfig
	%\square<500,250>[T(C)`T(E)`C`E;T(\gamma)`c_1`F`\gamma]
	\node TC(0,300)[T(C)]
	\node TE(650,300)[T(C \times_M E)]
	\node I(0,-300)[1]
	\node C(0,0)[C]
	\node E(650,0)[C \times_M E]
	\arrow|l|[C`TC;c_1]
	\arrow|a|[TC`TE;T(\beta)]
	\arrow|a|[C`E;\beta]
	\arrow|r|[E`TE;F]
	\arrow|l|[I`C;c_0]
	\arrow|b|[I`E;\<c_0,e_0\>]
\efig
\]
commute.  We now claim that the composite $\hat{\gamma} := \beta\pi_1: C \to E$ is the unique map satisfying conditions (i), (ii), (iii) in the statement of the theorem.  

For (i), 
	\[ c_0\hat{\gamma} = c_0 \beta \pi_1 = \<c_0,e_0\>\pi_1 = e_0. \]
For (ii), we first want to show that $\beta \pi_0 = 1_C$.  For this, consider the following diagram:
\[
\bfig
	%\square<500,250>[T(C)`T(E)`C`E;T(\gamma)`c_1`F`(\gamma)]
	\node TC(0,300)[T(C)]
	\node TE(650,300)[T(C \times_M E)]
	\node I(0,-300)[1]
	\node C(0,0)[C]
	\node E(650,0)[C \times_M E]
	\node C2(1300,0)[C]
	\node TC2(1300,300)[T(C)]
	\arrow|l|[C`TC;c_1]
	\arrow|a|[TC`TE;T(\beta)]
	\arrow|a|[C`E;\beta]
	\arrow|r|[E`TE;F]
	\arrow|l|[I`C;c_0]
	\arrow|a|[I`E;\<c_0,e_0\>]
	\arrow|a|[TE`TC2;T(\pi_0)]
	\arrow|a|[E`C2;\pi_0]
	\arrow|b|/{@{>}@/_1em/}/[I`C2;c_0]
	\arrow|r|[C2`TC2;c_1]
\efig
\]
The diagram commutes since the two left regions are the definition of $\beta$, the top right region commutes by definition of $F$, and the bottom right region is immediate.  Thus, $\beta \pi_0$ is a solution to the system $(c_1,c_0)$.  Note, however, that $1_C$ also solves this system.  Thus, by the uniqueness aspect of the curve object $C$, we must have $\beta \pi_0 = 1_C$.  To show (ii) is now immediate since
	\[ \hat{\gamma}q = \beta \pi_1 q = \beta \pi_0 \gamma = \gamma. \]

For (iii), consider
\begin{eqnarray*}
c_1T(\hat{\gamma})K 
& = & c_1T(\beta\pi_1)K \\
& = & c_1T(\beta)T(\pi_1)K \\
& = & \beta F T(\pi_1)K \mbox{ (by definition of $\beta$)} \\
& = & \beta \<\pi_0 c_1T(\gamma),\pi_1\>HK \mbox{ (by definition of $F$)} \\
& = & \beta \<\pi_0 c_1T(\gamma),\pi_1\> \pi_1 q 0_{\sf q} \mbox{ (by definition of a connection)} \\
& = & \beta \pi_1 q 0_{\sf q} \\
& = & \hat{\gamma} q 0_{\sf q} \\
& = & \gamma 0_{\sf q} \mbox{ (by (ii), proven above)} 
\end{eqnarray*}
	
Finally, we want to show that $\hat{\gamma}$ is the unique map with these properties.  Thus, suppose we have some $\gamma': C \to E$ satisfying (i), (ii), (iii).  Define $\beta' := \<1,\gamma'\>$.  If we can show $\beta'$ also solves the system $(F,\<c_0,e_0\>)$, that is, that
\[
\bfig
	%\square<500,250>[T(C)`T(E)`C`E;T(\gamma)`c_1`F`\gamma]
	\node TC(0,300)[T(C)]
	\node TE(650,300)[T(C \times_M E)]
	\node I(0,-300)[1]
	\node C(0,0)[C]
	\node E(650,0)[C \times_M E]
	\arrow|l|[C`TC;c_1]
	\arrow|a|[TC`TE;T(\beta')]
	\arrow|a|[C`E;\beta']
	\arrow|r|[E`TE;F]
	\arrow|l|[I`C;c_0]
	\arrow|b|[I`E;\<c_0,e_0\>]
\efig
\]
commutes, then the result will follow by the uniqueness of solutions since we would then have $\beta' = \beta$ and hence
	\[ \gamma' = \beta' \pi_1 = \beta \pi_1 = \hat{\gamma} \]
as required.

For the bottom triangle,
	\[ c_0 \beta' = c_0 \<1,\gamma'\> = \<c_0,e_0\> \]
by assumption on $\gamma'$.  

To show the square commutes, we are considering two maps into the pullback $T(C \times_M E)$, and hence it suffices to show the maps are equal when post-composed by $T(\pi_0)$ and $T(\pi_1)$.  For $T(\pi_0)$, 
	\[ c_1T(\beta')T(\pi_0) = c_1T(\beta'\pi_0) = c_1 \]
while
	\[ \beta' F T(\pi_0) = \beta ' \pi_0 c_1 = c_1 \]

For $T(\pi_1)$, consider
\begin{eqnarray*}
&   & c_1T(\beta')T(\pi_1) \\
& = & c_1T(\beta '\pi_1) \\
& = & c_1T(\gamma') \mbox{ (definition of $\beta$)} \\
& = & c_1T(\gamma')(\<K \lambda,p0\>T(+_{\sf q}) +_0 \<T(q),p\>H) \mbox{ (by definition of a connection)} \\
& = & \<c_1T(\gamma')K\lambda,c_1T(\gamma')p0\>T(+_{\sf q}) + \<c_1T(\gamma')T(q),c_1T(\gamma')p\>H \\
& = & \<\gamma 0_{\sf q}\lambda,c_1T(\gamma')p0\>T(+_{\sf q}) + \<c_1T(\gamma'q),c_1T(\gamma')p\>H \mbox{ (using (iii) for $\gamma'$)} \\
& = & \<\gamma 0_{\sf q}\lambda,c_1p\gamma'0)\>T(+_{\sf q}) + \<c_1T(\gamma),c_1T(\gamma')p\>H \mbox{ (using (ii) for $\gamma'$ and naturality of $p$)} \\
& = & \<\gamma 0T(0_{\sf q}),c_1p\gamma'0)\>T(+_{\sf q}) + \<c_1T(\gamma),c_1p\gamma'\>H \mbox{ (coherence of $\lambda$ and naturality of $p$)} \\
& = & \<\gamma 0T(0_{\sf q}),\gamma'0)\>T(+_{\sf q}) + \<c_1T(\gamma),\gamma'\>H \mbox{ ($c_1$ a vector field)} \\
& = & \gamma'0 + \<c_1T(\gamma),\gamma'\>H \mbox{ (adding an additive unit)} \\
& = & \<c_1T(\gamma),\gamma'\>H \mbox{ (adding an additive unit)} 
\end{eqnarray*}
while
	\[ \beta' FT(\pi_1) = \<1,\gamma'\>\<\pi_0c_1T(\gamma),\pi_1\>H = \<c_1T(\gamma),\gamma'\>H. \]
Thus $c_1T(\beta') = \beta' F$, and hence the diagram above commutes and the result is proven.  
\end{proof}

\subsection{Affine connections}

In this final section we look at two results about connections on a tangent bundle, that is, on affine connections.  

Define the twist map $\tau:= \<\pi_1,\pi_0\>:T_2(M) \to T_2(M)$.

\begin{lemma}\label{lemmaTau}
In any tangent category, $cU = U \tau$.  
\end{lemma}
\begin{proof}
\[ cU = c\<T(p),p\> = \<cT(p),cp\> = \<p,T(p)\> = U \tau. \]
\end{proof}

For affine connections $(K,H)$ one can give an alternative characterization of the torsion-free condition in terms of $H$.    

\begin{proposition}
In any tangent category, an affine connection $(K,H)$ is torsion-free if and only if $H c = \tau H$.
\end{proposition}
\begin{proof} The proof uses ideas drawn from \cite[Proposition 5]{bungeConnections}.

Assume $H c = \tau H$.  Consider
\begin{eqnarray*}
(K \ell +_1 p0) + UH & = & 1 \mbox{ (definition of a connection)} \\
c (K \ell c + p0 c)  +_1 c UH c & = & cc \mbox{ (apply the additive map $c$ to both ends of both equations)} \\
(cK\ell c + cp0c) +_1 cU\tau H & = & 1 \mbox{ (using the assumption and $c^2 = 1$)} \\
(cK\ell + T(p)T(0)) +_1 UH & = & 1 \mbox{ (by coherences of $\ell$ and $c$ and lemma \ref{lemmaTau})} \\
(cK\ell K + T(p)T(0)K) + UHK & = & K \mbox{ (applying $K$ to both sides, $K$ is additive for both additions)} \\
cK + 0 + 0 & = & K \mbox{ ($K$ is additive and by definition $H K = 0$ and $K\ell = 1$)} \\
cK & = & K
\end{eqnarray*}

For the other direction, assume $cK = K$.   Consider
\begin{eqnarray*}
(K \ell +_1 p0) + UH & = & 1 \mbox{ (definition of a connection)} \\
(cK\ell c + cp0c) +_1 cUHc & = & cc \mbox{ (apply the additive map $c$ to both ends of both equations)} \\
(K\ell + T(p)T(0)) +_1 cUHc & = & 1 \mbox{ (using the assumption and $c^2 = 1$)} \\
(HK \ell + HT(p)T(0)) +_1 H U \tau Hc & = & H \mbox{ (applying $H$ to both sides and using lemma \ref{lemmaTau})} \\
((\pi_0p 0 \ell) + \pi_0 T(0)) +_1 \tau Hc & = & H \mbox{ (assumptions on a connection)} \\
\pi_0 T(0) +_1 \tau Hc & = & H \mbox{ (addition of a zero term)} \\
\tau Hc & = & H \mbox{ (addition of a zero term)} \\
Hc & = & \tau H \mbox{ (applying $\tau$ to both sides)}
\end{eqnarray*}
as required.

%Consider the map $H_0 := \tau H c$.  We want to show that $H_0 = H$.  Note that $H_0$ is also a 			horizontal connection, as using lemma \ref{lemmaTau},
%	\[ H_0U = \tau H c U = \tau H U \tau = \tau \tau = 1. \]
%Thus $H_0$ has an associated connector $K_0$.  We now show that $K_0 = K$:
%\begin{eqnarray*}
%\<K_0,p\>\mu + UH_0 & = & 1 \\
%\<K_0,p\>\mu K + UH_0 K & = & K \\
%\<K_0,p\>\mu K + U\tau H c K & = & K \\
%\<K_0,p\>\pi_0 + U \tau H K & = & K \mbox{ (by assumption and lemma \ref{lemmaVConnector})} \\
%K_0 + U\tau0 & = & K \mbox{ (axiom on connection that $H K = 0$)} \\
%K_0 & = & K
%\end{eqnarray*}
%Thus $(K,H)$ and $(K,H_0)$ are both connections, so that by proposition \ref{propConnectDeterm}, $H = 		H_0$, as required.  
\end{proof}

Recall from theorem \ref{thmFinslerEquivalence} that any vertical connection $K$ has an associated Finsler descent $R: T(E) \to E_2$.  Then when $(K,H)$ is an affine connection on $M$, $R$ can be composed with $H$.  In this case, we can define an ``almost-complex'' structure on $M$ as follows.  This result is well-known in the category of smooth manifolds (eg., see \cite[definition 7.2.7(b)]{SLK}).  

\begin{proposition}
Let $\X$ be a tangent category with negatives, and $(K,H)$ an affine connection on $M$.  Then 
	\[  F := RH - Uv: T^2(M) \to T^2(M) \]
is an ``almost complex'' map; that is, it has the property that $FF = -1$.  
\end{proposition}
\begin{proof}
First, consider 
\begin{eqnarray*}
&   & vU \\
& = & \<\pi_0 \ell,\pi_1 0\>T(+)\<T(p),p\> \\
& = & \<\pi_0 \ell,\pi_1 0\> \<T(+p),T(+)p\> \\
& = & \<\pi_0 \ell,\pi_1 0\> \<T(\pi_1)T(p),p+ \> \\
& = & \<\pi_1 0T(p),\<\pi_0 \ell p,\pi_1\>+\> \\
& = & \<\pi_1 p0, \<\pi_0 p 0,\pi_1 \>+\> \\
& = & \<\pi_1 p0, \pi_1\> \\
& = & \pi_1 \<p0,1\>
\end{eqnarray*}
Thus
\begin{eqnarray*}
&   & UvUv \\
& = & \<T(p),p\>\pi_1\<p0,1\>\<\pi_0 \ell,\pi_1 0\>T(+) \\
& = & p\<\p0\ell, 0\>T(+) \\
& = & p\<p0T(0),0\>T(+) \\
& = & p0
\end{eqnarray*}
Recalling that $R = \<K,p\>$, we can also calculate
\begin{eqnarray*}
&   & RHRH \\
& = & RH\<K,p\>H \\
& = & R\<HK,Hp\>H \\
& = & \<K,p\>\<\pi_1 p0,\pi_1\>H \\
& = & \<pp0,p\>H \\
& = & p0 \mbox{ (since $H$ is linear).}
\end{eqnarray*}
We now calculate the desired quantity:
\begin{eqnarray*}
&   & FF \\
& = & (RH - Uv)(RH - Uv) \\
& = & RHRH - UvRH - RHUv + UvUv \mbox{ (all morphisms are additive)} \\
& = & p0 - UH - Rv + p0 \mbox{ (by above and lemma \ref{lemmaFinslerProps})} \\
& = & -UH - Rv \\
& = & -1
\end{eqnarray*}
by the definition of a connection.
\end{proof}

\newpage

%\bibliography{connections}

\end{document}